\documentclass[11pt, leqno]{article}
\usepackage{amsmath,amssymb,amsthm, epsfig}
\usepackage{hyperref}
\usepackage{amsmath}
\usepackage[pdftex]{color}
\usepackage{mathtools}
\usepackage{color}
\usepackage{ulem}
\usepackage{mathrsfs}
\usepackage{wasysym}
\usepackage{stackrel}
\usepackage{pgfplots}
\pgfplotsset{compat=newest}

\title{\Large{\bf Regularity to Thin Obstacle Problem in Orlicz spaces}}
\author{\it by \smallskip \\ Junior da Silva Bessa \footnote{\noindent Universidade Estadual de Campinas - UNICAMP. Departamento  de Matemática. Campinas - SP, Brazil. \noindent \texttt{E-mail address: \url{jbessa@unicamp.br}}},\,\,
Paulo Henryque da Costa Silva
\footnote{\noindent Universidade Federal do Cear\'{a}. Departamento de Matemática. Fortaleza - CE, Brazil. \noindent \texttt{E-mail address: \url{paulohenryque@alu.ufc.br}}}\,\, $\&$ \,\, Alan Pio Sousa \footnote{\noindent Universidade Federal do Cear\'{a}. Departamento de Matemática. Fortaleza - CE, Brazil. \noindent \texttt{E-mail address:  \url{alanpio@alu.ufc.br}}}
}

\newlength{\hchng}
\newlength{\vchng}
\def \dist {\mathrm{dist}}

\def \loc {\mathrm{loc}}

\newcommand{\defeq}{\mathrel{\mathop:}=}

\newtheorem{theorem}{Theorem}[section]
\newtheorem{lemma}[theorem]{Lemma}
\newtheorem{proposition}[theorem]{Proposition}

\theoremstyle{definition}

\newtheorem{definition}[theorem]{Definition}

\theoremstyle{remark}
\newtheorem{remark}[theorem]{Remark}

\numberwithin{equation}{section}

\newcommand{\intav}[1]{\mathchoice {\mathop{\vrule width 6pt height 3 pt depth  -2.5pt
\kern -8pt \intop}\nolimits_{\kern -6pt#1}} {\mathop{\vrule width
5pt height 3  pt depth -2.6pt \kern -6pt \intop}\nolimits_{#1}}
{\mathop{\vrule width 5pt height 3 pt depth -2.6pt \kern -6pt
\intop}\nolimits_{#1}} {\mathop{\vrule width 5pt height 3 pt depth
-2.6pt \kern -6pt \intop}\nolimits_{#1}}}

\begin{document}
\maketitle

\begin{abstract}
\noindent In this work, we establish regularity results for minimizers of the energy functional associated with the thin obstacle problem in Orlicz spaces. More precisely, we prove the Lipschitz continuity and the H\"older continuity of the gradient of minimizers. The analysis is based on techniques from De Giorgi's classical regularity theory. As a byproduct of our results, we also provide a characterization of the structure of the nodal sets of the minimizers.

\vspace{0.2cm}
	
\end{abstract}

\section{Introduction}

Variational problems featuring nonstandard growth conditions have emerged as a central theme in nonlinear analysis over the past two decades (see, e.g., \cite{rao-ren, harjulehto-hasto, diening-ruzicka}), particularly within the framework of Orlicz and Musielak–Orlicz spaces. These generalized function spaces arise naturally in the modeling of non-linear elasticity and electrorheological fluids \cite{ruzicka, diening-ruzicka}, as well as in image processing applications \cite{harjulehto-hasto}, where the classical quadratic growth associated with Dirichlet energy fails to capture intrinsic physical behavior.

In this work, we investigate the regularity of minimizers to a class of energy functionals of the form
\begin{equation}\label{eq1}
\mathcal{J}(u)=\int_{\mathrm{B}_{1}^{+}}G(|\nabla u(x)|)\,dx,
\end{equation}
where \( G:[0,+\infty)\to[0,+\infty) \) is a \( \mathrm{N} \)-function, over the admissible class
\[ 
\mathcal{G} = \left\{ v \in W^{1,G}(\mathrm{B}_{1}^{+}) : v = \varphi \ \text{on} \ (\partial \mathrm{B}_{1})^{+}, \quad v \geq 0 \ \text{on} \ \mathrm{T}_{1} \right\},
\]
with \( \mathrm{B}_{1}^{+} = \mathrm{B}_{1} \cap \mathbb{R}^{n}_{+} \), \( (\partial \mathrm{B}_{1})^{+} = \partial \mathrm{B}_{1} \cap \mathbb{R}^{n}_{+} \), and \( \mathrm{T}_{1} = \mathrm{B}_{1} \cap \{x_n = 0\} \), under suitable conditions on the function \(\varphi\).

The Euler--Lagrange system associated with the variational problem takes the form
\begin{equation}\label{euler}
\left\{
\begin{array}{rcll}
\Delta_{g} u &=& 0 & \text{in } \mathrm{B}^{+}_{1}, \\
u &=& \varphi & \text{on } (\partial \mathrm{B}_{1})^{+},\\
u &\geq& 0 & \text{on } \mathrm{T}_{1}, \\
\displaystyle \frac{g(|\nabla u|)}{|\nabla u|} u_{x_n} &=& 0 & \text{on } \mathrm{T}_{1} \cap \{ u > 0 \},\\
\displaystyle -\frac{g(|\nabla u|)}{|\nabla u|} \, u_{x_n} &\geq& 0 & \text{on } \mathrm{T}_{1}.
\end{array}
\right.
\end{equation}

This problem is known as the \textit{thin obstacle problem} (also referred to as the \textit{Signorini problem} or the \textit{boundary obstacle problem}), originally studied by the Italian mathematical physicist and civil engineer Antonio Signorini in the context of linear elasticity. It consists of finding the elastic equilibrium of a non-homogeneous, anisotropic elastic body resting on a rigid, frictionless surface and subjected only to its body forces (cf. \cite{Sig33,Sig59}). In other words, this problem can be seen as a variant of the classical obstacle problem, where the obstacle is imposed only on a set of ``lower dimension" (or ``thin", which explains the name of this problem). We recommend the book by Petrosyan et al. \cite{Schag10} for an introduction and further applications of the thin obstacle problem.

In the setting of regularity results for solutions to the Signorini problem, Caffarelli \cite{Caf79} studied the minimization problem \eqref{eq1} with \(G(t)=t^2\) in dimension \(n \ge 2\), obtaining $C^{1,\alpha}$ estimates for the minimizer, where \(\alpha \in (0,1/2]\) (see also \cite{Ura85}). Three decades later, Caffarelli and Athanasopoulos established the optimal, Caffarelli and Athanasopoulos established the optimal \(\displaystyle C^{1,\frac{1}{2}}\) regularity for the Laplace equation via monotonicity formulas in \cite{AthCaf06}. This result builds on the foundational work of Richardson, who first proved the optimal regularity in the two-dimensional case \((n=2)\) in \cite{Ric78}. Along these lines of Hölder gradient estimates in the thin obstacle problem, Andersson and Mikayelyan in \cite{AM} studied the regularity of minimizers of the functional \eqref{eq1} when \(G(t)=t^p\) for \(1<p<\infty\), employing a variant of De Giorgi's regularity theory.  Byun \textit{et al.} in \cite{BLOP} extended these estimates to the \(p(x)\)-Laplacian with variable exponent. 
Moreover, Fern\'andez--Real obtained \(C^{1,\alpha}\) estimates in the non-variational context for fully nonlinear elliptic operators \cite{Xavier}, and together with Ros--Otton studied the nonlocal scenario in \cite{RO18}. Subsequent progress on the thin obstacle problem can be found in \cite{DK25,RJ21,RS20,FS24,HT25,JPG24}.

Motivated by this, our work aims to establish regularity results for minimizers of the functional \eqref{eq1}. Specifically, under suitable hypotheses, we prove that the minimizers belong to the class \(C^{1,\gamma}\) for some \(\gamma \in (0,1)\). To this end, we first establish a Lipschitz regularity result for minimizers and develop a De Giorgi-type regularity theory for variational models with nonstandard growth. The lack of homogeneity in the associated Euler-Lagrange operator leads to substantial conceptual and technical differences from the classical \(p\)-homogeneous setting and requires a distinct analytical approach. In particular, our method extends the De Giorgi framework beyond homogeneous models. Provides a unified treatment that includes, as a special case, the energies \(G(t)=t^p\) considered in the previously cited works.

\subsection{Assumptions and Main Results}

We begin by collecting the basic notations used throughout the paper:

\begin{itemize}
  \item For each 
    $x=(x_1,\dots,x_n)\in\mathbb{R}^n$, we write 
    $x'=(x_1,\dots,x_{n-1})\in\mathbb{R}^{n-1}$.

  \item$\mathrm{B}_r(x)$ denotes the open ball of radius $r>0$ centered at $x \in \mathbb{R}^n$.
In particular, we have $\mathrm{B}_r = \mathrm{B}_r(0)$ and
$\mathrm{B}_r^{+} = \mathrm{B}_r \cap \mathbb{R}^n_{+}$.

  \item $\mathrm T_r=\{(x',0)\in\mathbb{R}^{n-1}\,:\,|x'|<r\}$ 
    is the flat portion of the boundary of $\mathrm B_r^+$.
    More generally, 
    $\mathrm T_r(x_0)=\mathrm T_r + x_0'$ for any $x_0'\in\mathbb{R}^{n-1}$.

  \item $(\partial \mathrm B_r)^+=\partial \mathrm B_r\cap\mathbb{R}^n_+$.

  \item Given a function $u$, its gradient is 
    $\nabla u=(u_{x_1},\dots,u_{x_n})$, and the tangential gradient is 
    $\nabla' u=(u_{x_1},\dots,u_{x_{n-1}})$.
\end{itemize}

We now state the structural conditions under which our main results are obtained:

\begin{itemize}
  \item[{\rm(\textbf{H1})}]\textbf{(Operator structure)} 
    Let $G\colon[0,\infty)\to[0,\infty)$ be an N-function satisfying Lieberman's conditions:
    \begin{itemize}
      \item[(i)] $G'=g$, where $g\in C^0([0,\infty))\cap C^1((0,\infty))$.
      \item[(ii)] There exist constants $0<\delta_0\le g_0$ such that
      \[
        \delta_0 \le \frac{t g'(t)}{g(t)} \le g_0 
        \qquad\forall\,t>0.
      \]
    \end{itemize}

  \item[{\rm(\textbf{H2})}]\textbf{(Data regularity)} 
    The boundary data $\varphi$ satisfies
    \[
      \varphi\in W^{1,G}(\mathrm B_1^+)\cap C^0(\overline{\mathrm B_1^+}),
      \quad\text{and}\quad
      \varphi\ge0 \quad\text{on}\quad \overline{\mathrm B_1^+}\cap\{x_n=0\}.
    \]
\end{itemize}

Regarding the scope of the functions $g$ that satisfy our assumptions, we observe that this class is quite broad. Consider $g(t)= t^a\log(bt+c)$ with $a, b, c>0$ which satisfies {\rm(\textbf{H1})} with $\delta_0=a$ and $g_0=a+1.$ Another interesting example is given by functions of the form $g(t)= at^p+ bt^q$ with $a, b, p, q>0,$ for which $\delta_0=\min\{p,q\}$ and $g_0=\max\{p,q\}.$

Moreover, any linear combination with positive coefficients of functions satisfying \rm(\textbf{H1}) also satisfies \rm(\textbf{H1}). 
In addition, if $g_1$ and $g_2$ satisfy condition \rm(\textbf{H1}) with constants $\delta_0^i$ and $g_0^i$, $i=1,2$, then the function $g = g_1 \cdot g_2$ satisfies \rm(\textbf{H1}) with $\delta_0 = \delta_0^1 + \delta_0^2$ and $g_0 = g_0^1 + g_0^2$, and the function $g(t) = g_1(g_2(t))$ satisfies \rm(\textbf{H1}) with $\delta_0 = \delta_0^1 \cdot \delta_0^2$ and $g_0 = g_0^1 \cdot g_0^2$.

Our main result is summarized in the following theorem.

\begin{theorem}\label{T1}
Assume the structural conditions {\rm(\textbf{H1})} and {\rm(\textbf{H2})} hold. Let \(u \in W^{1,G}(\mathrm B_1^+)\)
be a minimizer of the functional \eqref{eq1} over the admissible class \(\mathcal G\). Then there exists \(\gamma\in(0,1)\), depending only on \(n\), \(\delta_0\), and \(g_0\), such that \(u \in C^{1,\gamma}(\overline{\mathrm B_{1/2}^+})\).
\end{theorem}

We emphasize that, although our manuscript was partially inspired by the works \cite{AM,BLOP}, the presence of a nonhomogeneous operator associated with the functional, as well as the generality of the underlying function spaces, required a substantially different and nontrivial adaptation of the methods, which are not directly covered in the aforementioned references.

Due to the structure of the functional \eqref{eq1}, our results extend those in \cite{AM} and \cite{Caf79}, and, in certain contexts, draw a parallel with \cite{BLOP}, \cite{Ric78}, \cite{RO18}, and \cite{Xavier} within the line of research concerning regularity for solutions to thin obstacle problems.

\subsection{An application: Structure of the nodal sets}

As an application of our result, we analyze the structure of the nodal sets associated with minimizers of the functional \eqref{eq1}.  Specifically, for a solution $u$ to the thin obstacle problem \eqref{euler}, we define the nodal set of order $k \ge 1$ by
$$
\mathfrak{n}_{k}(u) = \left\{ x \in \overline{\mathrm{B}^{+}_{1/2}} \,\middle|\, D^\alpha u(x) = 0 \text{ for all } |\alpha| < k, \text{ and exist}  \ |\beta| = k \text{ such that } D^\beta u(x) \neq 0 \right\},
$$
where \(D^\alpha u(x)\) denotes the partial derivative of \(u\) at \(x\) associated to the multi-index \(\alpha\).

Nodal sets play a central role in describing the singular set in free boundary problems.  In the case $G(t)=t^2$, their structure has been extensively studied using monotonicity formulas for the Laplacian (see e.g.\ \cite[Chapter 9]{Schag10}).  More generally, Han proved in \cite{Han} that for solutions of linear elliptic equations of the form
\[
\mathcal L[u]=\sum_{|\nu|\le m}a_\nu(x)D^\nu u=f(x)\quad\text{in }\mathrm B_1,
\]
the nodal sets are \(C^{1,\beta}\)-manifolds for some \(\beta\in(0,1)\). In this direction, recently, Lian \cite{Lian} investigated the structure of nodal sets for operators in both divergence and non-divergence form within the parabolic framework.

From this perspective, using the estimates obtained from Theorem~\ref{T1}, we obtain a characterization of the set $\mathfrak{n}_{1}(u)$. The proof is analogous to \cite[Theorem 5.1]{Han}, and for this reason we omit it here.

\begin{theorem}[{\bf Structure of Nodal set}]\label{T2}
Under the same assumptions of Theorem \ref{T1} we have that
\[
\mathfrak{n}_{1}(u)=\bigcup_{j=0}^{n}\mathfrak{M}_{j},
\]
where each \(\mathfrak{M}_{j}\) is on a finite union of \(j\)-dimensional \(C^{1,\gamma}\)- manifolds, for some \(\gamma\in(0,1)\).
\end{theorem}

The remainder of the paper is organized as follows: Section \ref{Section2} is devoted to some definitions and preliminary results concerning the g-Laplacian operator spaces. In Section \ref{Section3}, we explore some properties of the minimizers, establishing the Lipschitz Regularity for the thin obstacle problem. Next, in Section \ref{Section4} we develop some regularity results in the context of De Giorgi's theory for the \(g\)-Laplacian. Finally, Section \ref{Section5} contains the proof of Theorem \ref{T1}.

\section{Preliminaries}\label{Section2}

We begin the section by recalling the definition and some elementary properties of the main class of functions considered in this work.

\begin{definition}\label{weightedorliczspaces}
The \textit{Orlicz space} $L^{G}(E)$ for be an $\mathrm{N}$-function $G$ and Lebesgue measurable set $E\subset \mathbb{R}^{n}$ is the set of all measurable functions $h$ in $E$ such that
\begin{eqnarray*}
\rho_{G}(h):=\int_{E}G(|h(x)|)dx<+\infty.
\end{eqnarray*}
In \(L^{G}(E)\) we can define the following Luxemburg norm
\begin{eqnarray*}
\|h\|_{L^{G}(E)}=:\inf\left\{t>0:\rho_{G}\left(\frac{h}{t}\right)\leq 1\right\}.
\end{eqnarray*}
Furthermore, the \textit{Orlicz-Sobolev space} $W^{k,G}(E)$ (for an integer $k\geq0$) is the set of all measurable functions $h$ in $E$ such that all distributional derivatives $D^{\alpha}h$, for multiindex $\alpha$ with length $|\alpha|=0,1,\cdots,k$ also belong to $L^{G}(E)$ with norm is given by
\begin{eqnarray*}
\|h\|_{W^{k,G}(E)}\defeq \sum_{j=1 }^{k}\|D^{j}h\|_{L^{G}(E)}.
\end{eqnarray*} 
\end{definition}
In our setting, the Lieberman's conditions in {\bf(H1)} implies that the \(L^{G}(E)\) and \(W^{k,G}(E)\) are the Banach reflexive spaces (see \cite[Theorem 2.1]{MW}).

Now, we list some basic properties on the Orlicz spaces.

\begin{lemma}[{\bf\cite[Lemma 2.3]{MW}}]\label{inclusion}
Let \(u\in L^{G}(\Omega)\) and assume that {\rm ({\bf H1})} holds. Then, there exists a positive constant \(\mathrm{C}=\mathrm{C}(\delta_{0},g_{0})\) such that 
\[
||u||_{L^G(\Omega)}\leq \mathrm{C}\max\left\{\rho_{G}(u)^{\frac{1}{1+\delta_{0}}},\rho_{G}(u)^{\frac{1}{1+g_{0}}}\right\}.
\]
\end{lemma}
The next result is a Poincaré-type inequality in the context of Orlicz-Sobolev spaces
\begin{lemma}[{\bf\cite[Lemma 2.4]{MW}}]\label{poincareinequality}
Let \(u\in W^{1,G}(\Omega)\) such that \(u=0\) on \(\partial \Omega\) and assume that {\rm({\bf H1})} is valid. Then, there exists a positive constant \(\mathrm{C}\) depending only on \(\operatorname{diam}(\Omega)\) such that 
\[
\rho_{G}(u)\leq \rho_{G}(\mathrm{C}|\nabla u|)
\]
\end{lemma}

\begin{remark}\label{trace}
The condition $v\geq 0$ on \(\mathrm{B}_{1}\cap \{x_{n}=0\}\) in the class $\mathcal{G}$ should be understood in the sense of the trace. Indeed, by Theorem 2.2 in \cite{MW} we have the continuous embedding $W^{1,G}(\mathrm{B}_{1}^+)\hookrightarrow W^{1,1+\delta_{0}}(\mathrm{B}_{1}^+).$ Moreover, by the Trace Compactness Theorem (Theorem 6.2 in \cite{Necas}), here exists a constant $r\geq1$ depending on $n$ and $\delta_0,$  such that the trace operator $T:W^{1,1+\delta_{0}}(\mathrm{B}_{1}^+)\to L^{r}(\partial \mathrm{B}_{1}^+)$ is compact. Therefore, the condition, $v\geq 0$ on $\mathrm{T}_{1}$ means that $T(v)(x)\geq 0$ for almost every point in $\mathrm{T}_{1}.$ 
\end{remark}

\subsection{g-Laplacian}

Associated with our functional \(\mathcal J\), we consider the classical \(g\)-Laplacian operator defined by
\begin{eqnarray}\label{glaplacian}
\Delta_{g}u=:\operatorname{div}\left(\frac{g(|\nabla u|)}{|\nabla u|}\nabla u\right).
\end{eqnarray} 
In this subsection, we recall its main properties and introduce the notion of weak solutions.
\begin{definition}
A function \(u\in W^{1,G}_{loc}(\Omega)\) is said a weak subsolution (respectively, supersolution)  of the equation
\[
\Delta_{g}u=0\,\,\, \text{in}\,\,\, \Omega,
\]
where \(\Omega\subset \mathbb{R}^{n}\) is an open set, if for any function \(0\leq\phi\in C^{\infty}_{0}(\Omega)\), we have
\[
\int_{\Omega}\frac{g(|\nabla u|)}{|\nabla u|}\nabla u\cdot\nabla\phi\  dx\leq 0 \, \left(\text{resp.} \int_{\Omega}\frac{g(|\nabla u|)}{|\nabla u|}\nabla u\cdot\nabla\phi\  dx\geq 0\right).
\]

Moreover, we say that \(u\in W^{1,G}_{\loc}(\Omega)\) is a weak solution of \(\Delta_{g}u=0\) (or \(u\) is a \(g\)-harmonic function) if \(u\) is weak sub and super solution.
\end{definition}
\begin{remark}
Under Lieberman's conditions (see {\bf (H1)}), the function \(g\) satisfies a two-sided growth control:
\[
\min\{t^{\delta_0},t^{g_0}\}\,g(s)\;\le\;g(ts)\;\le\;\max\{t^{\delta_0},t^{g_0}\}\,g(s),
\]
for all \(t,s>0\).
\end{remark}

\begin{remark}[Scaling and normalization]\label{scaling} The following results can be found in \cite[Remark 2.1]{JSousa}. 
Let \(u\) be a weak solution of \(\Delta_g u = 0\) in \(\mathrm B_\rho\). Define the rescaled function
\[
u_\rho(x) = \frac{u(\rho x)}{\rho}.
\]
Then \(u_\rho\) is weakly \(g\)-harmonic in \(\mathrm B_1\). Similarly, for a constant \(K>0\), set
\[
u^*(x)=\frac{u(x)}{K},\qquad
g^*(t)=\frac{g(Kt)}{g(K)}.
\]
One verifies that \(u^*\) satisfies \(\Delta_{g^*}u^*=0\) in \(\mathrm B_1\), where the N-function
\(G^*\), defined by \((G^*)' = g^*\), still obeys {\bf(H1)} with the same constants \(\delta_0\) and \(g_0\).
\end{remark}

\section{Some facts about thin obstacle problem in Orlicz spaces}\label{Section3}

In this section, we present a review of the minimizers of the functional \(\mathcal{J}\) in the admissible class \(\mathcal{G}\). More precisely, we establish the existence of minimizers in this class and develop some of their properties, in particular the Lipschitz regularity of minimizers.

We begin by stating a fundamental property of minimizers of the functional \(\mathcal{J}\), namely, that they are \(g\)-harmonic functions. The proof of this result is inspired by the ideas in \cite{MW} concerning minimizers in Orlicz spaces and is included here for the sake of completeness.

\begin{proposition}\label{prop3.2}
Let \(u\in W^{1,G}(\mathrm{B}^{+}_{1})\) be a minimizer of \(\mathcal{J}\) in \(\mathcal{G}\). Then \(u\) is a \(g\)-harmonic in \(\mathrm{B}^{+}_{1}\).
\end{proposition}
\begin{proof}
Consider \(0\leq\eta\in C^{\infty}_{0}(\mathrm{B}^{+}_{1})\). For all \(t\in[0,1]\) define \(u_{t}=u+t\eta\). We note that:
\begin{itemize}
\item[-] \(u_{t}\in W^{1,G}(\mathrm{B}_{1}^{+})\);
\item[-] In \((\partial \mathrm{B}_{1})^{+}\), \(u_{t}=u+t\eta=g+t\cdot 0=g\);
\item[-] In \(\mathrm{T}_{1}\), \(u_{t}=u+t\eta\geq u\geq 0\).
\end{itemize}
Thus, \(u_{t}\in \mathcal{G}\), for all \(t\in[0,1]\) and consequently, since \(u\) minimizes \(\mathcal{J}\) with respect to the class \(\mathcal{G}\), we have that
\begin{eqnarray}\label{cond1dolema3.2}
\frac{d}{dt}\Bigg|_{t=0}\mathcal{J}(u_{t})\geq 0.    
\end{eqnarray}
However,
\begin{eqnarray*}
\frac{d}{dt}\mathcal{J}(u_{t})&=&\int_{\mathrm{B}^{+}_{1}}\frac{d}{dt}G(|\nabla u+t\nabla \eta|)dx\\
&=& \int_{\mathrm{B}^{+}_{1}}\frac{d}{dt}g(|\nabla u+t\nabla \eta|)\frac{(\nabla u+t\nabla \eta)\cdot \nabla \eta}{|\nabla u+t\nabla \eta|}dx.
\end{eqnarray*}
Using this expression in \eqref{cond1dolema3.2} it follows that \(u\) is a sub \(g\)-harmonic in \(\mathrm{B}^{+}_{1}\). 

On the other hand, given \(\varepsilon > 0\), we have that \( u_{\varepsilon} = u - \varepsilon \eta \) is a competitor for minimizer of \(\mathcal{J}\) in \(\mathcal{G}\). In this case, by convexity of the N-function \(G\) and the minimality of \(u\), it holds that
\begin{eqnarray}\label{estdeuepsilon}
0&\leq&\varepsilon^{-1}(\mathcal{J}(u_{\varepsilon})-\mathcal{J}(u))\nonumber\\
&=&\varepsilon^{-1}\int_{\mathrm{B}^{+}_{1}}(G(|\nabla u-\varepsilon\nabla \eta|)-G(|\nabla u|))dx\nonumber\\
&\leq&\int_{\mathrm{B}^{+}_{1}}\left(-g(|\nabla u-\varepsilon\nabla \eta|)\frac{(\nabla u-\varepsilon\nabla \eta)\cdot \nabla \eta}{|\nabla u-\varepsilon\nabla\eta|}\right)dx.
\end{eqnarray}
Taking \(\varepsilon \to 0\) in \eqref{estdeuepsilon}, we obtain that
\begin{eqnarray*}
0\leq\int_{\mathrm{B}^{+}_{1}}\left(-g(|\nabla u|)\frac{\nabla u\cdot \nabla \eta}{|\nabla u|}\right)dx\Longrightarrow\int_{\mathrm{B}^{+}_{1}}\frac{g(|\nabla u|)}{|\nabla u|}\nabla u\cdot \nabla \eta dx\leq 0.
\end{eqnarray*}
Thus, \(u\) is a sub \(g\)-harmonic function in \(\mathrm{B}^{+}_{1}\).
\end{proof}

\subsection{Lipschitz regularity of minimizers}

Now, we proceed to develop a compactness tool for the thin obstacle problem. To this end, we will prove that minimizers of $\mathcal{J}$ are Lipschitz continuous. The main idea to achieve such regularity is to extend the original problem to a classical obstacle problem (cf. \cite{AM, Xavier}). Specifically, let $u$ be a minimizer of $\mathcal{J}$ over the class $\mathcal{G}$. We may consider the even extensions $\tilde{u}$ and $\tilde{\varphi}$ of $u$ and $\varphi$, respectively, over $\mathrm{T}_{1}$:

\begin{eqnarray}\label{par}
\tilde{u}(x)=:\left\{
\begin{array}{rclcl}
u(x',x_{n}),&\text{if}& x_{n}\geq 0\\
u(x',-x_{n}),&\text{if}& x_{n}\leq 0
\end{array}
\right. \,\,\text{and}\,\, \tilde{\varphi}(x)=:\left\{
\begin{array}{rclcl}
\varphi(x',x_{n}),& \text{if}& x_{n}\geq 0\\
\varphi(x',-x_{n}),& \text{if}& x_{n}\leq 0.
\end{array}
\right.
\end{eqnarray}

By the Gluing Sobolev Functions (see \cite[Proposition 8.1]{BM21}) we have that \(\tilde{u}\in W^{1,G}(\mathrm{B}_{1})\) and $\tilde{\varphi}\in W^{1,G}(\mathrm{B}_{1})\cap C^0(\overline{\mathrm{B}_{1}})$. Moreover, we note that \(\tilde{u}\) is a minimizer of the functional
\begin{eqnarray}\label{funcionaldoobstaculo}
\tilde{\mathcal{J}}(v)=\int_{\mathrm{B}_{1}}G(|\nabla v(x)|)dx
\end{eqnarray}
over the admissible set
\[
\tilde{\mathcal{G}}=\{w\in W^{1,G}(\mathrm{B}_{1}); w=\tilde{\varphi}\,\,\,\text{on}\,\,\, \partial\mathrm{B}_{1}\,\,\, \text{and}\,\,\, w\geq 0\,\,\,\text{in}\,\,\, \mathrm{T}_{1}\}.
\]
Really, let \( v \in \tilde{\mathcal{G}}\) be a competitor for the functional \(\tilde{\mathcal{J}}\). Defining \( \tilde{v}(x',x_n) := v(x',-x_n) \), we see that both \( v\vert_{\mathrm{B}_1^+} \) and \( \tilde{v}\vert_{\mathrm{B}_1^+} \) belong to the set \( \mathcal{G} \). So, by the minimality of \(u\), it holds  
\begin{eqnarray*}
\tilde{\mathcal{J}}(v)=\int_{\mathrm{B}_1} G(|\nabla v|)dx &=& \int_{\mathrm{B}_1^+} G(|\nabla v|)dx + \int_{\mathrm{B}_1^+} G(|\nabla \tilde{v}|)dx\\
&\geq& 2\int_{\mathrm{B}^{+}_{1}}G(|\nabla u|)dx=\tilde{\mathcal{J}}(\tilde{u}),
\end{eqnarray*}
since in \(\mathrm{B^{-}_{1}}\), \(|\nabla \tilde{u}|\) is equal to \(|\nabla u|\) in \(\mathrm{B}^{+}_{1}\). Thus, we conclude that \( \tilde{u} \in \tilde{\mathcal{G}} \) is a minimizer for the functional \(\tilde{\mathcal{J}}\).

With this approach, starting from the obstacle problem, we can prove the existence of a minimizer of the functional $\mathcal{J}$ in the class $\mathcal{G}$.

 \begin{proposition}
Assume that \({\rm({\bf H1})}-{\rm({\bf H2})}\) holds. Then, there exists a minimizer \(u\in \mathcal{G}\) of the functional \(\mathcal{J}\).
 \end{proposition}
 
\begin{proof}
First, observe that the set \(\mathcal{G}\) is nonempty. Indeed, due to the properties of \(\varphi\) given in ({\bf H2}), we have that \(\varphi \in \mathcal{G}\) and \(0\leq\mathcal{J}(\varphi)<+\infty\) (since \(\varphi\in W^{1,G}(\mathrm{B}^{+}_{1})\)).
Let \((u^{k})_{k\in\mathbb{N}}\subset \mathcal{G}\) be a minimizing sequence for the functional \(\mathcal{J}\). Considering the even extension of each function, we obtain a sequence $(\tilde{u}^k)_{k\in \mathbb{N}}$ in the class $\mathcal{\tilde{G}}.$ Since the sequence $(u^k)_{k\in \mathbb{N}}$ is a minimizer, it follows that the sequence $\left(\tilde{\mathcal{J}}(\tilde{u}^k)\right)_{k\in \mathbb{N}}$ is bounded, and by Lemma \ref{poincareinequality}, we have
$$\int_{\mathrm{B}_{1}}G(|\tilde{u}^k-\tilde{\varphi}|)\leq \int_{\mathrm{B}_{1}}G(\mathrm{C}|\nabla(\tilde{u}^k-\tilde{\varphi})|). $$
 Thus, there exists a constant $\mathrm{C}>0$ such that for all $k\in \mathbb{N}$
 $$\int_{\mathrm{B}_{1}}(G(|\tilde{u}^k|) + G(|\nabla \tilde{u}^k|))\,dx \leq \mathrm{C}. $$
 Therefore, since each $\tilde{u}_k$ is an even extension of $u_k,$ by Lemma \ref{inclusion}, we see that for all $k\in \mathbb{N}$ we have $||u^k||_{W^{1,G}(\mathrm{B}_{1}^+)}\leq \mathrm{C}.$ Consequently, there exists a function $u_{\infty}\in W^{1,G}(\mathrm{B}_{1}^+)$ such that, up to the extraction of a subsequence $u^k \rightharpoonup u_\infty$ in $W^{1,G}(\mathrm{B}_{1}^+).$ By the Remark \ref{trace}, we have that $u^k \rightharpoonup u_\infty$ in $W^{1,1+\delta_{0}}(\mathrm{B}_{1}^+).$  Moreover, by the compact embedding $W^{1,1+\delta_{0}}(\mathrm{B}_{1})\hookrightarrow L^{1+\delta_{0}}(\mathrm{B}_{1})$ we have $u^k\to u_{\infty}$ a.e. in $\mathrm{B}_{1}^+,$ and given that—as discussed in Remark \ref{trace}—the trace operator $T:W^{1,1+\delta_{0}}(\mathrm{B}_{1}^+)\to L^{r}(\partial \mathrm{B}_{1}^+)$ is compact, it follows that $u_\infty=\varphi$ on $(\partial \mathrm{B}_{1})^+$ and $u_\infty\geq 0$ on $\mathrm{T}_{1}.$ Thus, we conclude that $u_\infty\in \mathcal{G}.$ 
 Now observe that the convexity of $G$ implies that
  $$\int_{\mathrm{B}_{1}^+}G(|\nabla u^k|)\,dx \geq  \int_{\mathrm{B}_{1}^+}G(|\nabla u_\infty|)\,dx + \int_{\mathrm{B}_{1}^+}\frac{g(|\nabla u_\infty|)}{|\nabla u_\infty|}\nabla u_\infty\cdot(\nabla u^k- \nabla u_\infty)\,dx, $$
 and since $u^k \rightharpoonup u_\infty$ in $W^{1,G}(\mathrm{B}_{1}^+),$ it follows that
 $$\int_{\mathrm{B}_{1}^+}G(|\nabla u_\infty|)\,dx\leq \liminf_{k\to \infty} \int_{\mathrm{B}_{1}^+}G(|\nabla u^k|)\,dx.$$
  Hence, since $(u^k)_{k\in \mathbb{N}}$ is a minimizer sequence in $\mathcal{G},$ it follows that $u_\infty\in \mathcal{G}$ is a minimizer of the same functional $\mathcal{J}.$

\end{proof}

From the minimizer \( \tilde{u} \), we can construct a function \( \psi \) that serves as an obstacle for \( \tilde{u} \).
\begin{lemma}\label{obst}
There exists a function \( \psi \in C^0(\mathrm{B}_1) \cap W^{1,G}(\mathrm{B}_1) \) that vanishes on \( \mathrm{T}_{1} \) for which \( \tilde{u} \) is the solution to the classical obstacle problem in \( \mathrm{B}_1 \) with obstacle \( \psi \).
\end{lemma}
\begin{proof}
Let \( \psi_1 \) and \( \psi_2 \) be the solutions in \( \mathrm{B}_1^+ \) and \( \mathrm{B}_1^- \), respectively, of the following Dirichlet problems:
\begin{eqnarray*}
\left \{
\begin{array}{rclcl}
\Delta_{g}\psi_{1} & = & 0 & \text{in} & \mathrm{B}^{+}_{1}\\
\psi_1 & = & -\|u\|_{L^{\infty}(\mathrm{B}^{+}_{1})} & \text{on} & (\partial \mathrm{B}_1)^+\\
\psi_1 & = & 0 & \text{on} & \overline{\mathrm{T}_1}\\
\end{array}
\right.
\,\,\,\,\text{and}\,\,\,\,
\left \{
\begin{array}{rclcl}
\Delta_{g}\psi_{2} & = & 0 & \text{in} & \mathrm{B}^{-}_{1}\\
\psi_2 & = & -\|u\|_{L^{\infty}(\mathrm{B}^{+}_{1})} & \text{on} & (\partial \mathrm{B}_1)^{-}\\
\psi_2 & = & 0 & \text{on} & \overline{\mathrm{T}_1}.\\
\end{array}
\right.
\end{eqnarray*}
The boundary regularity theorem proved in \cite[Theorem 1.1]{BS25} ensures that  
\( \psi_i \in C^{1,\alpha}(\overline{\mathrm{B}_{r}^{\pm}}) \) for \( i=1,2 \) and \( r \in (0,1) \), with the estimate  
\[
\|\psi_i\|_{C^{1,\alpha}(\overline{\mathrm{B}_r^{\pm}})} \leq \mathrm{C} \cdot \|u\|_{L^{\infty}(\mathrm{B}_1^+)},
\]
where \( \mathrm{C} = \mathrm{C}(n,\delta_{0},g_{0},g(1),r) > 0 \).  

Now, we can invoke the Comparison Principle \cite[Lemma 2.8]{MW}, we have \( u \geq \psi_i \) in \( \mathrm{B}_1^{\pm} \).  
Defining \( \psi \) as the function  
\[
\psi(x',x_n) =
\begin{cases}
\psi_1(x',x_n) \quad \text{in} \quad \overline{\mathrm{B}_1^+}; \\
\psi_2(x',x_n) \quad \text{in} \quad \overline{\mathrm{B}_1^-}.
\end{cases}
\]
We see that the extension \( \tilde{u} \) satisfies \( \tilde{u} \geq \psi \) in \( \mathrm{B}_1 \). Moreover, by \cite[Proposition 8.1]{BM21}, we have \( \psi \in W^{1,G}(\mathrm{B}_1) \cap C^0(\mathrm{B}_1) \).  

Finally, let us prove that \( \tilde{u} \) is a solution to the classical obstacle problem associated with the \( g \)-Laplacian operator with obstacle \( \psi \).  To this end, we consider the variational problem of minimizing the functional \(\tilde{\mathcal{J}}\) over the set  
\[
\mathcal{G}_{\psi} := \{v \in W^{1,G}(\mathrm{B}_1); v = \tilde{f} \quad \text{on} \quad \partial \mathrm{B}_1, \,\,\, v \geq \psi \quad \text{in} \quad \mathrm{B}_1\}.
\]
Since for any \( v \in \mathcal{G}_{\psi} \) we have \( v \geq \psi = 0 \) on \( \mathrm{T}_1 \), it follows that \( v \in \tilde{\mathcal{G}} \), and by minimality, we obtain  
\[
\tilde{\mathcal{J}}(\tilde{u}) = \int_{\mathrm{B}_1} G(|\nabla\tilde{u}|)dx \leq \int_{\mathrm{B}_1} G(|\nabla v|)dx = \tilde{\mathcal{J}}(v).
\]
Therefore, \( \tilde{u} \) is a minimizer for the classical obstacle problem.  
\end{proof}
\begin{remark}
Combining this result with Hölder regularity for the minimizer of the classical obstacle problem (cf. \cite[Theorem 5.8]{Kar21}), we conclude that the minimizer \( \tilde{u}\) of \(\tilde{\mathcal{J}}\) over set \(\tilde{\mathcal{G}}\) is locally Hölder continuous in \( \mathrm{B}_1\).    
\end{remark}

In the variational problem \eqref{funcionaldoobstaculo} we have the following result.
\begin{proposition}
Let \(\tilde{u}\) be a minimizer of the functional \(\tilde{\mathcal{J}}\) over the set \(\tilde{\mathcal{G}}\). Then, \(\tilde{u}\) is solution to
\begin{eqnarray*}
\left \{
\begin{array}{rclcl}
\Delta_{g}\tilde{u} & = & 0 & \text{in} & \mathrm{B}_{1}\setminus \{\tilde{u}=0\}\cap\{x_{n}=0\}\\
\Delta_{g}\tilde{u} & \leq & 0 & \text{in} & \mathrm{B}_{1}\\
\tilde{u} & \geq & 0 & \text{on} & \mathrm{T}_{1}\\
\tilde{u} & = & \tilde{\varphi} & \text{on} & \partial\mathrm{B}_{1}.\\
\end{array}
\right.
 \end{eqnarray*}      
\end{proposition}
\begin{proof}
As in Proposition \ref{prop3.2}, we can verify that $\tilde{u}$ is a supersolution in \(\mathrm{B}_{1}\) and is \(g\)-harmonic in $\mathrm{B}_1^+$ ensuring that $\Delta_g\tilde{u}=0$ in the set $\mathrm{B}_1\setminus\{x_n=0\}$. It remains to verify that $\tilde{u}$ is $g$-harmonic in $\mathrm{B}_{1}\cap\{\tilde{u}>0\}$. For this, we note that for all  $w\in \tilde{\mathcal{G}}$, we have  
$$\int_{\mathrm{B}_1}\frac{g(|\nabla\tilde{u}|)}{|\nabla\tilde{u}|}\nabla\tilde{u}\cdot\nabla(\tilde{u}-w)dx \geq 0.
$$  
 
Now, since the set $\{u>0\}\cap \mathrm{B}_1$ is open, given a point $x_0\in\{u>0\}\cap\{x_n=0\}$, there exists $\mathrm{B}_r(x_0)\subset\subset \{u>0\}\cap \mathrm{B}_1$. Let $\xi\in \mathrm{C}_0^{\infty}(\mathrm{B}_r(x_0))$ and define $w=\tilde{u}+t\xi$. Note that $w\in \tilde{\mathcal{G}}$, since in $\mathrm{B}_r(x_0)^{c}\cap\{x_n=0\}$ we have $w=\tilde{u}\geq 0$. Thus, for sufficiently small $|t|$, we also have $w=\tilde{u}+t\xi >0$ in $\mathrm{B}_r(x_0)\cap\{x_n=0\}$. Therefore, $w\geq 0$ in $\{x_n=0\}$, and therefore $w\in \tilde{\mathcal{G}}$. In this form, by the previous observation, we obtain
$$
t\int_{\mathrm{B}_r(x_0)}\frac{g(|\nabla\tilde{u}|)}{|\nabla\tilde{u}|}\nabla\tilde{u}\cdot\nabla\xi dx\,\geq 0.
$$  
Since the above inequality holds for both positive and negative values of $t$ of sufficiently small magnitude, we conclude that
$$
\int_{\mathrm{B}_r(x_0)}\frac{g(|\nabla\tilde{u}|)}{|\nabla\tilde{u}|}\nabla\tilde{u}\cdot\nabla\xi\,dx=0. 
$$  
That is, $\Delta_g\tilde{u}=0$ in $\mathrm{B}_r(x_0).$  
\end{proof}
With these preliminary observations, we can prove the main result of this section on the solutions for the thin obstacle problem.

\begin{proposition}\label{Lipschitzregularity}
Let $u\in W^{1,G}(\mathrm{B}^{+}_{1})$ be a minimizer of the functional \ref{eq1}. Then, $u$ is Lipschitz continuous in \(\mathrm{B}^{+}_{3/4}\). Specifically, we have the following estimate  
 $$\|u\|_{C^{0,1}(\mathrm{B}_{3/4}^+)}\leq \mathrm{\mathrm{C}_1}\|u\|_{L^{\infty}(\mathrm{B}_1^+)},
 $$  
 where $\mathrm{\mathrm{C}_1}$ is a positive constant depending only on \(n\), \(g_0\), \(\delta_0\) and \(g(1)\).  
\end{proposition}
\begin{proof}
We consider \(\tilde{u}\), the even extension of \(u\), which is a minimizer of the variational problem \ref{funcionaldoobstaculo}. Given \(\tau\in (0,1)\), let \(\psi\in W^{1,G}(\mathrm{B}_1)\cap C^0(\mathrm{B}_1)\) be the obstacle obtained in Lemma \ref{obst} by gluing the functions \(\psi_i\in C^{1,\alpha}(\overline{\mathrm{B}_{\tau}^{\pm}})\), which satisfy, due to the boundary regularity for \(g\)-harmonic functions \cite[Theorem 1.1]{BS25} with such estimate  
\[
\|\psi_i\|_{C^{1,\alpha}(\mathrm{B}_\tau^{\pm})} \leq \mathrm{C}\|u\|_{L^{\infty}(\mathrm{B}_1^+)},
\]
where \(\mathrm{C}=\mathrm{C}(n,\delta_0,g_0,g(1),\tau)>0\).  

To obtain a constant depending only on the dimension and the PDE parameters, we choose \(\tau=\sqrt{7/8}\).  
Using \cite[Proposition 8.1]{BM21} once again, we see that \(\psi\) is Lipschitz in \(\mathrm{B}_{\tau}\) with the estimate  
\[
\|\psi\|_{C^{0,1}(\mathrm{B}_{\tau})}\leq \mathrm{C}\|u\|_{L^{\infty}(\mathrm{B}_1^+)}=:\mathrm{C}\cdot \mathrm{K}.
\]  
Thus, for any \(x_0\in \mathrm{B}_{\tau}\) and \(0<r<\left(1-\tau\right)\tau\), we have  
\begin{equation} \label{eq 3}
\sup_{\mathrm{B}_r(x_0)}|\psi(x)-\psi(x_0)| \leq \mathrm{C} \cdot \mathrm{K}r,
\end{equation}  
since in this case, \(\overline{\mathrm{B}_r(x_0)}\subset \mathrm{B}_{\tau}\).

Now, we obtain an estimate analogous to (\ref{eq 3}) for \(\tilde{u}\) when \(x_0\in \{\tilde{u}=\psi\}\). Indeed, in this case, using that \(\tilde{u}\geq \psi\), we obtain  
\begin{equation}\label{eq 4}
\inf_{\mathrm{B}_r(x_0)}\left(\tilde{u}(x)-\tilde{u}(x_0)\right)\geq -\mathrm{C} \cdot \mathrm{K}r.
\end{equation}
Define the function \(h(x)= \tilde{u}(x)-\tilde{u}(x_0) + \mathrm{CK}r\). From (\ref{eq 4}), we have \(h\geq 0\) in \(\mathrm{B}_r(x_0)\). On the other hand, from (\ref{eq 3}), we obtain
\[
h(x)\leq 2\mathrm{CK}r\quad\text{in}\quad \mathrm{B}_r(x_0)\cap\{\tilde{u}=\psi\}.
\]
Moreover, \(h\) is obviously a supersolution. Now, let \(\tilde{h}\) be the solution to the following Dirichlet problem:  
\begin{equation*}
\left\{
\begin{array}{rclcl}
\Delta_{g}\tilde{h}&=& 0& \mbox{in} &   \mathrm{B}_{r}(x_{0}); \\
\tilde{h} &=& h &\mbox{on}& \partial\mathrm{B}_{r}(x_{0}).\\
\end{array}
\right.
\end{equation*}
By the Comparison Principle, we get \(\tilde{h}\leq h\) in \(\mathrm{B}_r(x_0)\). Furthermore, since \(h\geq 0\) in \(\mathrm{B}_r(x_0)\), it follows from \cite[Theorem 2.2]{BM21} that \(\tilde{h}\geq 0\) in \(\overline{\mathrm{B}_r(x_0)}\). Thus, we obtain  
\begin{equation}\label{eq 5}
h<\tilde{h} + 2\mathrm{CK}r\quad\text{on}\quad \partial \mathrm{B}_r(x_0)\quad\text{and}\quad h\leq\tilde{h} + 2\mathrm{CK}r\quad\text{in}\quad \mathrm{B}_r(x_0)\cap\{\tilde{u}=\psi\}.
\end{equation}
Therefore,
\begin{equation}\label{eq 6}
h\leq \tilde{h} + 2\mathrm{CK}r\quad\text{in}\quad \mathrm{B}_r(x_0).
\end{equation}
On the other hand, we have \(0\leq \tilde{h}(x_0)\leq h(x_0)=\mathrm{CK}r\), and using Harnack's inequality (cf. \cite[Lemma 3.4]{B18}), we find that, for a new constant \(\mathrm{C}_1=\mathrm{C}_1(n,\delta,g_0)>0\),  
\[
\tilde{h}\leq \mathrm{C}_1\cdot \mathrm{CK}r\quad\text{in}\quad \overline{\mathrm{B}_{r/2}(x_0)}.
\]
Thus, combining this estimate with (\ref{eq 6}), we obtain  
\[
\tilde{u}(x)-\tilde{u}(x_0)\leq \mathrm{C}_1\cdot \mathrm{CK}r\quad\text{in}\quad \overline{\mathrm{B}_{r/2}(x_0)}.
\]
Therefore, combining this last inequality with (\ref{eq 4}), we conclude that, for a new constant \(\mathrm{C}=\mathrm{C}(n,\delta_0,g_{0},g(1))>0\), we have  
\begin{equation}\label{lip}
\sup_{\mathrm{B}_{r/2}(x_0)}|\tilde{u}(x)-\tilde{u}(x_0)| \leq \mathrm{CK}\frac{r}{2}.
\end{equation}
Now, let \(x\in \mathrm{B}_{\tau}\) and \(x_0\in \{\tilde{u}=\psi\}\cap \mathrm{B}_{\tau}\). We then consider two cases:
\begin{itemize}
\item [(i)] \(|x-x_0|<\frac{1}{2}\left(1-\tau\right)\tau\)\\
In this case, we can take \(r=2|x-x_0|\) in (\ref{lip}) and obtain  
\[
|\tilde{u}(x)-\tilde{u}(x_0)| \leq \mathrm{CK}|x-x_0|.
\]  
\item[(ii)] \(|x-x_0|\geq \frac{1}{2}\left(1-\tau\right)\tau\)\\
We have the obvious inequality  
\[
|\tilde{u}(x)-\tilde{u}(x_0)| \leq \frac{4}{\left(1-\tau\right)\tau}\cdot\|\tilde{u}\|_{L^{\infty}(\mathrm{B}_1)}\cdot|x-x_0|.
\] 
\end{itemize} 

 We may assume the existence of at least one point
 \(x_0\in \{\tilde{u}=0\}\cap\{x_n=0\}\cap \mathrm{B}_{5/16}\). Otherwise, we would have  
\[
\Delta_g\tilde{u}=0\quad\text{in}\quad \mathrm{B}_{5/16},
\]
and by the interior \(C^{1,\alpha}\) regularity estimate for \(g\)-harmonic functions in \(\mathrm{B}_{1/4}\subset \mathrm{B}_{5/16}\) (see Theorem 2.3 in \cite{BM21}), the result follows.

Given \(x, y \in \mathrm{B}_{1/4}\), let \(r = |x - y|\) and define  
\[
\rho := \min\left\lbrace \text{dist}(x, \{\tilde{u}=\psi\}), \text{dist}(y, \{\tilde{u}=\psi\})\right\rbrace.
\]  
Let \(x^*, y^* \in \{\tilde{u}=\psi\}\) such that \(\text{dist}(x, \{\tilde{u}=\psi\}) = |x - x^*|\) and \(\text{dist}(y, \{\tilde{u}=\psi\}) = |y - y^*|\), and assume, without loss of generality, that \(\rho = |x - x^*|\). Using the triangle inequality, we get  
\[
|y - x^*| \leq |y - x| + |x - x^*| \leq r + \rho,
\]  
so  
\[
|y - y^*| \leq |y - x^*| \leq r + \rho,
\]  
and therefore  
\[
|x^* - y^*| \leq 2(r + \rho).
\]  

Note that since there exists at least one point \(x_0 \in \{\tilde{u} = 0\} \cap \{x_n = 0\} \cap \mathrm{B}_{5/16}\) and \(\{\tilde{u} = 0\} \cap \{x_n = 0\} \subset \{\tilde{u}=\psi\}\), it follows that \(x^*, y^* \in \mathrm{B}_{7/8}\). 

Thus, we can treat the following cases:
\begin{itemize}
\item If \(\rho \leq 4r\), then  
\[
\begin{aligned}
|\tilde{u}(x) - \tilde{u}(y)| &\leq |\tilde{u}(x) - \tilde{u}(x^*)| + |\tilde{u}(y) - \tilde{u}(y^*)| + |\psi(x^*) - \psi(y^*)| \\
&\leq \mathrm{CK}\rho + \mathrm{CK}(r + \rho) + 2\mathrm{CK}(r + \rho).
\end{aligned}
\]
\item If \(\rho > 4r\), then in this case we have \(y \in \mathrm{B}_{\rho/4}(x)\). Moreover, note that  
\[
\mathrm{B}_{\rho/2}(x) \subset \mathrm{B}_1 \setminus \{\tilde{u}=\psi\}\subset \mathrm{B}_1 \setminus \{\tilde{u} = 0\} \cap \{x_n = 0\},
\]  
and since \(\tilde{u}\) is \(g\)-harmonic in this set, in particular, we have an interior estimate for the gradient (see for example in \cite[Lemma 2.7]{MW})  
\[
\sup_{\mathrm{B}_{\rho/4}} |\nabla \tilde{u}| \leq \frac{\mathrm{C}}{\rho} \cdot \operatornamewithlimits{osc}_{\mathrm{B}_{\rho/2}} \tilde{u}.
\]  

Furthermore, since for all \(z \in \mathrm{B}_{\rho/2}(x)\), we have  
\[
\tilde{u}(x^*) - |\tilde{u}(z) - \tilde{u}(x^*)| \leq \tilde{u}(z) \leq |\tilde{u}(z) - \tilde{u}(x^*)| + \tilde{u}(x^*),
\]  
it follows that \(\displaystyle\sup_{\mathrm{B}_{\rho/2}} \tilde{u} \leq \mathrm{CK}\rho + \psi(x^*)\) and \(\displaystyle\inf_{\mathrm{B}_{\rho/2}} \tilde{u} \geq -\mathrm{CK}\rho + \psi(x^*)\).  

Thus,  
\[
\operatornamewithlimits{osc}_{\mathrm{B}_{\rho/2}} \tilde{u} \leq 2\mathrm{CK}\rho.
\]
\end{itemize}
Thus, in any of these cases, we obtain for any \(x, y \in \mathrm{B}_{1/4}\)  
\[
|\tilde{u}(x) - \tilde{u}(y)| \leq \mathrm{CK}|x - y|.
\]

Finally, observe that the argument presented can be applied at any ball \(\mathrm{B}_r(z) \subset \mathrm{B}_1\) centered at a point \(z \in \mathrm{B}_{3/4} \cap \{x_n = 0\}\). By using a covering argument together with the Lipschitz estimate in the interior of \(\mathrm{B}_1^+\), for a new constant \(\mathrm{C_1} = \mathrm{C_1}(n, \delta_{0}, g_0, g(1)) > 0\) (possibly larger than the previous one), the following estimate holds  
\[
\|u\|_{C^{0,1}(\mathrm{B}_{3/4}^+)} \leq \mathrm{C_1}\|u\|_{L^\infty(\mathrm{B}_1^+)}
\]
as desired.
\end{proof}  
\section{Review of De Giorgi’s Theory on Regularity}\label{Section4}

In this part, we develop some De Giorgi-type regularity lemmas for a minimizer of the functional \eqref{eq1}. We begin with the following two technical lemmas. The first one is the following result, whose proof can be found in \cite[Lemma 6.1]{Giusti}.

\begin{lemma}\label{pre1}
Let $z(t)$ be a bounded non-negative function in the interval $[\rho, R].$ Assume that for $\rho\leq t\leq s\leq R$ we have
$$
z(t)\leq \left( \mathfrak{A}(t-s)^{-\alpha}+ \mathfrak{B}(t-s)^{-\beta} +\mathfrak{C} \right)+ \eta\cdot z(s),
$$
with $\mathfrak{A}, \mathfrak{B}, \mathfrak{C}\geq 0,\,\,\alpha>\beta>0$ and $\eta\in (0,1).$ Then,
$$
z(\rho)\leq \mathrm{C}(\alpha,\eta)\left( \mathfrak{A}(R-\rho)^{-\alpha}+ \mathfrak{B}(R-\rho)^{-\beta} +\mathfrak{C} \right), 
$$ where $\mathrm{C}(\alpha,\eta)= (1-\lambda)^{-\alpha}(1-\eta\lambda^{-\alpha})^{-1}$ and $\lambda\in (0,1)$ is chosen so that $\lambda^{-\alpha}\eta<1.$
\end{lemma}

The following lemma is a slight variation of Lemma 7.1 from \cite{Giusti}, adapted to suit our specific case.

\begin{lemma}\label{pre2}
Let $\alpha>0$ and let $\{\psi_i\}_{i=1}^{\infty}$ be a sequence of real positive numbers such that 
$$
\psi_{i+1}\leq \mathrm{C}\mathrm{B}^i\psi_i^{1+\alpha},
$$
with $\mathrm{C}>0$ and $\mathrm{B}>1.$ If $\psi_1\leq \mathrm{C}^{-\frac{1}{\alpha}}\mathrm{B}^{-\frac{1+\alpha}{\alpha^2}},$ we have 
\begin{equation}\label{de1}
\psi_i\leq \mathrm{B}^{-\frac{i-1}{\alpha}}\psi_1.
\end{equation}
In particular, $\lim\limits_{i\to +\infty}\psi_i=0.$
\end{lemma}

\begin{proof}
We proceed by induction. The case $i=1$ is obviously true. Assume now that it holds for $i.$ We have 
$$
\psi_{i+1}\leq \mathrm{C}\mathrm{B}^i\left(\mathrm{B}^{-\frac{i-1}{\alpha}}\psi_1 \right)^{1+\alpha} \leq \left(\mathrm{C}\mathrm{B}^{\frac{1+\alpha}{\alpha}}\psi_1^\alpha\right)\mathrm{B}^{-\frac{i}{\alpha}}\psi_1, $$
and \eqref{de1} follows for $i+1.$
\end{proof}

\subsection{Regularity lemmas}

Let $u$ be a minimizer of \eqref{eq1} in the admissible class $\mathcal{G}$. Since we know that $u\in C^{0,1}(\mathrm{T}_{3/4}),$ for each $m \in \{1, 2, \dots, n-1\}$, define the function $v = u_{x_m}.$ Then we have $v=0$ in $\mathrm{T}_{3/4}\cap\{u=0\}.$ Now  we need once again to turn our attention to the even extension of the minimizer function $u$ defined in \eqref{par}. For  $x_n>0$ we have
$$\tilde{u}_{x_m}(x',-x_n)= u_{x_m}(x',x_n).$$
Therefore, if we define $\tilde{v}:=\tilde{u}_{x_m}$ in $\mathrm{B}_1$ we have 
\begin{equation}\label{Eq7}
    \tilde{v}=v\quad\text{in}\quad \mathrm{B}_1^+\quad\text{and}\quad\tilde{v}(x',x_n)=v(x',-x_n)\quad\text{in}\quad \mathrm{B}_1^-.
\end{equation}
Since $\tilde{u}$ is $g-$harmonic in $\mathrm{B}_1\setminus\{u=0\}\cap\{x_n=0\},$ it follows that for any $k>0,$ we have $\tilde{u}\in C^{1,\alpha}_{loc}(\{\tilde{v}\geq k\}\cap \mathrm{B}_1).$  Moreover, in the set $\{\tilde{v}\geq k\}\cap \mathrm{B}_{3/4},$ we have
 \begin{equation}\label{quoc}
     \frac{g(k)}{C_1||u||_{L^\infty(\mathrm{B}_1^+)}}\leq \frac{g(|\nabla\tilde{u}|)}{|\nabla\tilde{u}|}\leq \frac{g\left(C_1||u||_{L^\infty(\mathrm{B}_1^+)}\right)}{k}. 
 \end{equation}
Therefore, Schauder theory implies that $\tilde{u}\in C^{\infty}_{loc}(\{\tilde{v}\geq k\}\cap \mathrm{B}_{3/4}).$ Thus, we may differentiate the equation
 $$ \operatorname{div}\left(\frac{g(|\nabla\tilde{u}|)}{|\nabla\tilde{u}|}\nabla\tilde{u}\right)=0\quad\text{in}\quad \{\tilde{v}\geq k\}\cap \mathrm{B}_{3/4}, $$
 we would get
 $$\left(a^{ij}\frac{g(|\nabla\tilde{u}|)}{|\nabla\tilde{u}|}\tilde{u}_{x_jx_m} \right)_{x_i}=0\quad\text{in}\quad \{\tilde{v}\geq k\}\cap \mathrm{B}_{3/4},$$
where the coefficients are given by
\[
a^{ij}(\nabla u) =
\begin{cases}
\begin{alignedat}{2}
    &u_{x_i} u_{x_j} \left( \frac{g'(|\nabla u|)}{g(|\nabla u|) |\nabla u|} - \frac{1}{|\nabla u|^2} \right) + \delta_{ij} 
    &&\quad \text{if} \quad |\nabla u| \neq 0; \\
    &\delta_{ij} 
    &&\quad \text{if} \quad |\nabla u| = 0.
\end{alignedat}
\end{cases}
\]
Direct calculations show that these coefficients satisfy for all $\xi\in \mathbb{R}^n\setminus\{0\}$
$$\min\{1,\delta_0\}|\xi|^2\leq a^{ij}\xi_i\xi_j \leq \max\{1,g_0\}|\xi|^2. $$
In what follows, for the sake of simplicity, we will use the following notations:
 $$\lambda=\min\{1,\delta_0\},\quad \Lambda=\max\{1,g_0\},\quad\text{and}\quad \mathrm{A}(k)=\{x\in \mathrm{B}_1^+;\,\,v(x)\geq k\}.$$
 This motivates us to consider test functions $\phi\in C_0^{\infty}(\{\tilde{v}\geq k\}\cap \mathrm{B}_{3/4}),$ and integration by parts yields
 $$0=\int_{\mathrm{B}_{3/4}}\frac{g(|\nabla\tilde{u}|)}{|\nabla\tilde{u}|}\nabla\tilde{u}\cdot\nabla\phi_{x_m} = -\int_{\mathrm{B}_{3/4}}\frac{g(|\nabla\tilde{u}|)}{|\nabla\tilde{u}|}a^{ij}\tilde{u}_{x_jx_m}\phi_{x_i}. $$
 A standard approximation argument shows that this equation remains valid for any test function $\phi\in W^{1,G}(\mathrm{B}_t)$ with support contained in $\{\tilde{v}\geq k\}\cap \mathrm{B}_{3/4}$ and therefore we get
\begin{equation}\label{div}
\int_{\mathrm{B}_{3/4}}\frac{g(|\nabla \tilde{u}|)}{|\nabla \tilde{u}|}a^{ij}\tilde{v}_{x_j}\phi_{x_i}=0.
\end{equation}
This is the key property satisfied by the minimizers, from which all the results in this section follow.

We begin by studying the gradient $\nabla v$ in the sets $\mathrm{A}(k)$. In particular, we establish in these sets a decay estimate involving $(v - k)^{+}$, which is the content of the next result. More precisely, it is an estimate that controls the energy functional in the $L^2$ norm. In order to make the dependence of the constants involved explicit, we will provide proofs for some of them.

\begin{lemma}\label{degeo1}
Let $v=u_{x_m}$ for $m=1,2,...,n-1$ where $u$ is a minimizer of the functional \eqref{eq1} on the admissible class \(\mathcal{G}\). Then for any $k>0$ and $0<s< t<\frac{3}{4}$ we have
$$\int_{\mathrm{A}(k)\cap \mathrm{B}_s^+}|\nabla v|^2 \leq \frac{\mathrm{C}_2}{(t-s)^2}\frac{2\tilde{k}\mathrm{C}_3 + \mathrm{C}_3^2}{\tilde{k}^2}\int_{\mathrm{A}(k)\cap \mathrm{B}_t^+}((v-k)^+)^2,$$
where $\tilde{k}=\min\{k^{1+\delta_0}, k^{1+g_0}\}$
$$
\mathrm{C}_3=\frac{\Lambda n^2}{\lambda}\max\{(\mathrm{C}_1||u||_{L^\infty(\mathrm{B}^{+}_{1})})^{1+g_0}, (\mathrm{C}_1||u||_{L^\infty(\mathrm{B}^{+}_{1})})^{1+\delta_0}\},
$$
and for $\lambda\in (0,1)$ such that $\frac{\mathrm{C}_3}{\mathrm{C}_3+\tilde{k}}\lambda^{-2}<1,$ we have 
$$\mathrm{C}_2= (1-\lambda)^{-1}\left(1- \frac{\mathrm{C}_3}{\mathrm{C}_3+\tilde{k}}\lambda^{-2}\right)^{-1}.$$
\end{lemma}

\begin{proof}

We start by taking a function $\eta \in \mathrm{C}_0^{\infty}(\mathrm{B}_t)$ such that $0 \leq \eta \leq 1$, $\eta \equiv 1$ in $\mathrm{B}_s$, and $\|\nabla \eta\|_{L^\infty(\mathrm{B}_t)} \leq \frac{2}{t-s}$. We define the test function $\phi=\eta(\tilde{v}-k)^+$. As $\operatorname{supp}(\phi)\subset \mathrm{B}_t\cap \{\tilde{v}\geq k\}$ we conclude from \ref{div} that
$$
0=\int_{\mathrm{B}_t}\frac{g(|\nabla\tilde{u}|)}{|\nabla\tilde{u}|}a^{ij}\tilde{v}_{x_j}\eta_{x_i}(\tilde{v}-k)^+ + \int_{\mathrm{B}_t}\frac{g(|\nabla\tilde{u}|)}{|\nabla\tilde{u}|}a^{ij}\tilde{v}_{x_j}(\tilde{v}-k)^+_{x_i}\eta,
$$
which implies
\begin{equation}\label{Eq2}
\int_{\mathrm{B}_t\cap\{\tilde{v}\geq k\}}\frac{g(|\nabla\tilde{u}|)}{|\nabla\tilde{u}|}a^{ij}\tilde{v}_{x_j}\tilde{v}_{x_i}\eta = -\int_{\mathrm{B}_t}\frac{g(|\nabla\tilde{u}|)}{|\nabla\tilde{u}|}a^{ij}\tilde{v}_{x_j}\eta_{x_i}(\tilde{v}-k)^+.
\end{equation}
For the left-hand side of \eqref{Eq2} we have
$$\int_{\mathrm{B}_t\cap \{\tilde{v}\geq k\}}\frac{g(|\nabla\tilde{u}|)}{|\nabla\tilde{u}|}a^{ij}\tilde{v}_{x_j}\tilde{v}_{x_i}\eta \geq \lambda\int_{\mathrm{B}_s\cap\{\tilde{v}\geq k\}}\frac{g(|\nabla\tilde{u}|)}{|\nabla\tilde{u}|}|\nabla \tilde{v}|^2. $$
Meanwhile, on the right-hand side of \eqref{Eq2} we estimate
\begin{align*}
\left\vert \int_{\mathrm{B}_t}\frac{g(|\nabla\tilde{u}|)}{|\nabla\tilde{u}|}a^{ij}\tilde{v}_{x_j}\eta_{x_i}(\tilde{v}-k)^+ \right\vert 
    &\leq \Lambda n^2 \int_{\mathrm{B}_t}\frac{g(|\nabla\tilde{u}|)}{|\nabla\tilde{u}|}|\nabla \tilde{v}||\nabla\eta|(\tilde{v}-k)^+.
\end{align*}
Therefore, combining the two estimates, we obtain:
\begin{equation}\label{Eq3}
 \lambda\int_{\mathrm{B}_s\cap \{\tilde{v}\geq k\}}\frac{g(|\nabla\tilde{u}|)}{|\nabla\tilde{u}|}|\nabla \tilde{v}|^2 \leq \Lambda n^2 \int_{\mathrm{B}_t}\frac{g(|\nabla\tilde{u}|)}{|\nabla\tilde{u}|}|\nabla \tilde{v}||\nabla\eta|(\tilde{v}-k)^+. 
\end{equation}
 Now, we use that estimate \eqref{quoc} holds in $\mathrm{B}_{3/4}\cap \{\tilde{v}\geq k\},$ then, from \eqref{Eq3}, we get that
\begin{eqnarray}\label{estofnablav}
\int_{\mathrm{B}_s\cap \{\tilde{v}\geq k\}}|\nabla\tilde{v}|^2\leq \frac{\mathrm{C}_3}{\tilde{k}}\int_{\mathrm{B}_t\cap \{\tilde{v}\geq k\}}|\nabla \tilde{v}|\cdot|\nabla\eta|(\tilde{v}-k)^+,
\end{eqnarray} 
where 
$$
\tilde{k}=\min\{k^{1+\delta_0}, k^{1+g_0}\}\,\, \text{and}\,\, \mathrm{C}_3=\frac{\Lambda n^2}{\lambda}\max\{(\mathrm{C}_1||u||_{L^\infty(\mathrm{B}^{+}_{1})})^{1+g_0}, (\mathrm{C}_1||u||_{L^\infty(\mathrm{B}^{+}_{1})})^{1+\delta_0}\}.$$
Hereafter, given \(\epsilon>0\) to be chosen a \textit{posteriori}, we can apply the Cauchy inequality with epsilon for obtain in \eqref{estofnablav}that  
$$\int_{\mathrm{B}_s\cap \{\tilde{v}\geq k\}}|\nabla \tilde{v}|^2\leq  \frac{\mathrm{C}_3}{\tilde{k}}\epsilon\int_{\mathrm{B}_t\cap \{\tilde{v}\geq k\}}|\nabla \tilde{v}|^2 + \frac{\mathrm{C}_3}{\tilde{k}4\epsilon}\int_{\mathrm{B}_t\cap \{\tilde{v}\geq k\}}|\nabla\eta|^2((\tilde{v}-k)^+)^2. $$
Now we add in bouth sides
$$\frac{\mathrm{C}_3}{\tilde{k}}\epsilon\int_{\mathrm{B}_s\cap \{\tilde{v}\geq k\}}|\nabla \tilde{v}|^2,$$
and multiply by $$\frac{\tilde{k}}{\tilde{k}+\epsilon \mathrm{C}_3},$$
to obtain the following estimate
\begin{equation}\label{Eq4}
    \int_{\mathrm{B}_s\cap \{\tilde{v}\geq k\}}|\nabla \tilde{v}|^2\leq \frac{2\mathrm{C}_3}{\tilde{k}+ \epsilon \mathrm{C}_3}\epsilon\int_{\mathrm{B}_t\cap \{\tilde{v}\geq k\}}|\nabla \tilde{v}|^2 + \frac{\mathrm{C}_3}{(\tilde{k}+\epsilon \mathrm{C}_3)4\epsilon}\int_{\mathrm{B}_t\cap \{\tilde{v}\geq k\}}|\nabla\eta|^2((\tilde{v}-k)^+)^2.
\end{equation}
We choose $\epsilon>0$ such that 
$$
\frac{2\mathrm{C}_3}{\tilde{k}+ \epsilon \mathrm{C}_3}\epsilon= \frac{\mathrm{C}_3}{\tilde{k}+\mathrm{C}_3},
$$
specifically,
$$
\epsilon= \frac{\mathrm{C}_3\tilde{k}}{2\mathrm{C}_3\tilde{k}+ \mathrm{C}_3^2}.
$$
Therefore, inequality \eqref{Eq4} with this choice of $\epsilon,$ the fact that $||\nabla\eta||_{L^\infty(\mathrm{B}_{t})}\leq \frac{2}{t-s}$ and some obvious estimates yields that
$$\int_{\mathrm{B}_s\cap \{\tilde{v}\geq k\}}|\nabla \tilde{v}|^2\leq \frac{\mathrm{C}_3}{\tilde{k}+\mathrm{C}_3}\int_{\mathrm{B}_t\cap \{\tilde{v}\geq k\}}|\nabla \tilde{v}|^2 + \frac{2\tilde{k}\mathrm{C}_3 + \mathrm{C}_3^2}{\tilde{k}^2}\frac{1}{(t-s)^2}\int_{\mathrm{B}_t\cap \{\tilde{v}\geq k\}}((\tilde{v}-k)^+)^2. $$

Now we apply Lemma \ref{pre1} with $\alpha=2,\,\,\,\eta= \frac{\mathrm{C}_3}{\tilde{k}+\mathrm{C}_3},$
$$z(t)=\int_{\mathrm{B}_t\cap \{\tilde{v}\geq k\}}|\nabla \tilde{v}|^2,\quad \mathfrak{A}=\frac{2\tilde{k}\mathrm{C}_3+ \mathrm{C}_3^2}{\tilde{k}^2}\int_{\mathrm{B}_t\cap \{\tilde{v}\geq k\}}((\tilde{v}-k)^+)^2,\quad \mathfrak{B}=\mathfrak{C}=0, $$
and this completes the proof of the lemma using \eqref{Eq7}.
\end{proof}

As a consequence of the previous result, we obtain a control over the region where $v$ is large.
\begin{lemma} \label{degeo2}
Let $v=u_{x_m}$ for $m=1,2,...,n-1$ where $u$ is a minimizer of the functional \eqref{eq1} on the admissible class \(\mathcal{G}\). If $v\leq 1,$ and $|\mathrm{A}(k_0)\cap \mathrm{B}_r^+|\leq \varrho|\mathrm{B}_r^+|$ for some $\varrho\in (0,1)$ then for every large constant $C>0$ there exists a constant $k\in (0,1)$ such that
$$
|\mathrm{A}(k)\cap \mathrm{B}_r^+| < \frac{|\mathrm{B}_r^+|}{C}. 
$$
\end{lemma}

\begin{proof}
    Having established Lemma \ref{degeo1}, the proof of this result follows precisely the same steps as Lemma 3.4 in \cite{AM}, and is therefore omitted.
\end{proof}

 The final result is an estimate controlling the supremum of the partial derivatives $v=u_{x_m}.$ 
\begin{lemma}\label{degeo3}
    Let $v=u_{x_m}$ for $m=1,2,...,n-1$ where $u$ is a minimizer of the functional \eqref{eq1} on the admissible class \(\mathcal{G}\). Then for $k_0>0$ and $0<\rho<\frac{3}{4}$ 
    $$\sup_{\mathrm{B}_{\frac{\rho}{2}}^+}v \leq \mathrm{C}_4\left(\frac{1}{\rho^n}\int_{\mathrm{B}_\rho^+}((v-k_0)^+)^2 \right)^{\frac{1}{2}}\left(\frac{|A(k_0)\cap \mathrm{B}_\rho^+|}{\rho^n}\right)^{\frac{\alpha}{2}} + k_0, $$
    where $\alpha(\alpha +1)=\frac{2}{n}$ and $\mathrm{C}_4>0$ depends only on  $n, \mathrm{C}_2$ and $\mathrm{C}_3.$
\end{lemma}

\begin{proof}
  First, observe that the function $\tilde{v}-k_0$ satisfies equation \eqref{div} for all test functions supported in the set  $\{\tilde{v} \geq k_0\}\cap B_{3/4}.$
  Therefore, by considering $v-k_0$ instead of $v$, we may assume without loss of generality that $k_0 = 0$.
  Furthermore, up to scaling and normalization arguments as discussed in Remark \ref{scaling}, we may, without loss of generality, assume that $\rho=3/4,$  and
  \begin{equation}\label{Eq9}
      \left(\int_{\mathrm{B}_{3/4}^+}(v^+)^2 \right)^{\frac{1}{2}}|\{v\geq 0\}\cap \mathrm{B}_{3/4}^+|^{\frac{\alpha}{2}} = 1.
  \end{equation}
Now, let $\frac{1}{2}<\sigma<\tau\leq \frac{3}{4}$ and $\eta\in \mathrm{C}_0^{\infty}(\mathrm{B}_{\frac{\sigma+\tau}{2}})$ such that
  $$0\leq \eta\leq 1,\quad \eta\equiv 1\,\,\,\,\text{in}\,\,\,\mathrm{B}_\sigma,\quad\text{and}\quad ||\nabla\eta||_{L^{\infty}} \leq \frac{4}{\tau-\sigma}.  $$
  Considering $\tilde{v}$ the extension of $v,$ we define $w=\eta(\tilde{v}-k)^+$ for $k>0,$ then we have that $\operatorname{supp}(w)\subset \mathrm{A}(k)\cap \mathrm{B}_{\frac{\sigma+\tau}{2}}$ and
  \begin{align*}
      \int_{\mathrm{A}(k)\cap \mathrm{B}_\sigma}((\tilde{v}-k)^+)^2 &\leq \int_{\mathrm{A}(k)\cap \mathrm{B}_{\frac{\sigma+\tau}{2}}}w^2 \leq \left(\int_{\mathrm{A}(k)\cap \mathrm{B}_{\frac{\sigma+\tau}{2}}}w^{\frac{2n}{n-2}} \right)^{\frac{n-2}{n}}|\mathrm{A}(k)\cap \mathrm{B}_{\frac{\sigma+\tau}{2}}|^{\frac{2}{n}}\\
      &\leq \mathrm{C}(n)\left(\int_{\mathrm{A}(k)\cap \mathrm{B}_{\frac{\sigma+\tau}{2}}}|\nabla w|^2\right)|\mathrm{A}(k)\cap \mathrm{B}_{\frac{\sigma+\tau}{2}}|^{\frac{2}{n}}.
  \end{align*}
  Since for $0<h\leq k$ we have $\mathrm{A}(k)\subset \mathrm{A}(h),$ then
  \begin{equation}\label{Eq5}
     \int_{\mathrm{A}(k)\cap \mathrm{B}_\sigma}((\tilde{v}-k)^+)^2 \leq \mathrm{C}(n)\left(\int_{\mathrm{A}(h)\cap \mathrm{B}_{\frac{\sigma+\tau}{2}}}|\nabla w|^2\right)|\mathrm{A}(k)\cap \mathrm{B}_{\frac{\sigma+\tau}{2}}|^{\frac{2}{n}}. 
  \end{equation}
  For the first term in the right-hand side, we have
  \begin{equation*}
      \int_{\mathrm{A}(h)\cap \mathrm{B}_{\frac{\sigma+\tau}{2}}}|\nabla w|^2 \leq \frac{8}{(\tau-\sigma)^2}\int_{\mathrm{A}(h)\cap \mathrm{B}_{\frac{\sigma+\tau}{2}}}((\tilde{v}-h)^+)^2 + 2\int_{\mathrm{A}(h)\cap \mathrm{B}_{\frac{\sigma+\tau}{2}}}|\nabla\tilde{v}|^2.
  \end{equation*}
  Since $\tilde{v}$ is the extension of $v$ and satisfies \eqref{Eq7} we get
  \begin{equation}\label{Eq8}
      \int_{\mathrm{A}(h)\cap \mathrm{B}_{\frac{\sigma+\tau}{2}}}|\nabla w|^2 \leq \frac{16}{(\tau-\sigma)^2}\int_{\mathrm{A}(h)\cap \mathrm{B}_{\frac{\sigma+\tau}{2}}^+}((v-h)^+)^2 + 4\int_{\mathrm{A}(h)\cap \mathrm{B}_{\frac{\sigma+\tau}{2}}^+}|\nabla v|^2.
  \end{equation}
  Now, let $d>0$ be a constant to be determined later, and consider $k\geq h>0$ such that $d\leq h\leq k\leq 2d.$ Applying the Lemma \ref{degeo1} in this configuration we have
  \begin{equation}
      \int_{\mathrm{A}(h)\cap \mathrm{B}_{\frac{\sigma+\tau}{2}}^+}|\nabla v|^2 \leq \frac{4\mathrm{C}_2}{(\tau-\sigma)^2}\frac{2^{g_0+2}\tilde{d}\mathrm{C}_3+ \mathrm{C}_3^2}{\tilde{d}^2}\int_{\mathrm{A}(h)\cap \mathrm{B}_\tau^+}((v-h)^+)^2,
  \end{equation}
 where $\tilde{d}=\min\{d^{1+g_0},d^{1+\delta_0}\}$ and the constant $\mathrm{C}_2$ in given in Lemma \ref{degeo1}
  for $\lambda\in (0,1)$ such that $\frac{\mathrm{C}_3}{\tilde{d}+\mathrm{C}_3}\lambda^{-2}<1.$
 If we assume that $d\geq 1$ then it is sufficient take $\lambda\in (0,1)$ such that $\frac{\mathrm{C}_3}{1+\mathrm{C}_3}\lambda^{-2}<1,$ and in this case for 
 $$\mathrm{C}_2=(1-\lambda)^{-2}\left(1-\frac{\mathrm{C}_3}{\mathrm{C}_3+1}\lambda^{-2} \right)^{-1}, $$
 we have
 \begin{equation}\label{Eq6}
 \int_{\mathrm{A}(h)\cap \mathrm{B}_{\frac{\sigma+\tau}{2}}^+}|\nabla v|^2 \leq \frac{4\mathrm{C}_2}{(\tau-\sigma)^2}\frac{2^{g_0+2}\mathrm{C}_3+ \mathrm{C}_3^2}{\tilde{d}}\int_{\mathrm{A}(h)\cap \mathrm{B}_\tau^+}((v-h)^+)^2.
 \end{equation}
  Using the estimate \eqref{Eq6} in \eqref{Eq8} we get
 $$ \int_{\mathrm{A}(h)\cap \mathrm{B}_{\frac{\sigma+\tau}{2}}}|\nabla w|^2 \leq \frac{1}{(\tau-\sigma)^2}\left[\frac{16\tilde{d}+ 16\mathrm{C}_2(2^{g_0+2}\mathrm{C}_3+\mathrm{C}_3^2) }{\tilde{d}} \right]\int_{\mathrm{A}(h)\cap \mathrm{B}_\tau^+}((v-h)^+)^2. $$
 Then by \eqref{Eq5} we have
 \begin{equation*}\label{ALt1}
 \int_{\mathrm{A}(k)\cap \mathrm{B}_\sigma^+}((v-k)^+)^2 \leq \frac{\mathrm{C}(n)}{(\tau-\sigma)^2}\left[\frac{8\tilde{d}+ 8\mathrm{C}_2(2^{g_0+2}\mathrm{C}_3+\mathrm{C}_3^2) }{\tilde{d}}\right]|\mathrm{A}(k)\cap \mathrm{B}_{\frac{\sigma+\tau}{2}}|^{\frac{2}{n}}\int_{\mathrm{A}(h)\cap \mathrm{B}_\tau^+}((v-h)^+)^2.
 \end{equation*}
On the other hand, since
 $$|\mathrm{A}(k)\cap \mathrm{B}_\tau^+|^\alpha \leq \frac{1}{(k-h)^{2\alpha}}\left(\int_{\mathrm{A}(h)\cap \mathrm{B}_\tau^+}((v-h)^+)^2 \right)^\alpha, $$
 multiplying this inequality with \eqref{Eq7} and using that $\tilde{d}=d^{1+\delta_0}$ (recall that \(d\geq 1\)) we get the following
 \begin{equation}\label{ALT2}
    |\mathrm{A}(k)\cap \mathrm{B}_\tau^+|^\alpha\int_{\mathrm{A}(k)\cap \mathrm{B}_\sigma^+}((v-k)^+)^2 \leq  
 \frac{\mathrm{C}(n)\bar{\mathrm{C}}|\mathrm{A}(k)\cap \mathrm{B}_{\frac{\sigma+\tau}{2}}|^{\frac{2}{n}}}{(\tau-\sigma)^2(k-h)^{2\alpha}}\left(\int_{\mathrm{A}(h)\cap \mathrm{B}_\tau^+}((v-h)^+)^2\right)^{1+\alpha},
 \end{equation}
 where
 \[
\bar{\mathrm{C}}=\left[\frac{8d^{1+\delta_0}+ 8\mathrm{C}_2(2^{g_0+2}\mathrm{C}_3+\mathrm{C}_3^2) }{d^{1+\delta_0}} \right].
 \]
 Now let $k_i=2d(1-2^{-i}),\,\,\,\sigma_i=\frac{1}{2}(1+2^{-i})$ for all $i\in \mathbb{N},$ and
 $$\psi_i= |\mathrm{A}(k_i)\cap \mathrm{B}_{\sigma_i}^+|^\alpha\int_{\mathrm{A}(k_i)\cap \mathrm{B}_{\sigma_i}^+}((v-k_i)^+)^2. $$
 Then, inequality \eqref{ALT2} implies, with $\sigma=\sigma_{i+1},\,\,\tau=\sigma_i,\,\,k=k_{i+1},\,\,h=k_i$ and $\alpha^2+\alpha=\frac{2}{n},$ that

 \begin{equation*}
\psi_{i+1} \leq 16(4^{1+\alpha})^i\mathrm{C}(n)\frac{8d^{1+\delta_0} + 8\mathrm{C}_2(2^{g_0+2}\mathrm{C}_3+ \mathrm{C}_3^2)}{d^{1+\delta_0 +2\alpha}}\cdot\psi_i.
 \end{equation*}
 By increasing the constant $\mathrm{C}(n)$ and using the that $d\geq 1$, it follows that
 \begin{equation}
\psi_{i+1}\leq \mathrm{C}(n)\frac{1+ \mathrm{C}_2(2^{g_0+2}\mathrm{C}_3+\mathrm{C}_3^2)}{d^{2\alpha}}(4^{1+\alpha})^i\cdot\psi_i.  
 \end{equation}
 Next, with the aim of applying Lemma \ref{pre2}, we take $\mathrm{B}=4^{1+\alpha}$ and
 $$\mathrm{C}=\mathrm{C}(n)\frac{1+ \mathrm{C}_2(2^{g_0+2}\mathrm{C}_3+\mathrm{C}_3^2)}{d^{2\alpha}},$$
 with this, we require that
 $$\psi_1 \leq \mathrm{C}^{-\frac{1}{\alpha}}\mathrm{B}^{-\frac{1+\alpha}{\alpha^2}}. $$
 Since hypothesis \eqref{Eq9} ensures that $\psi_1\leq 1,$ It suffices to have
  $$1 \leq \mathrm{C}^{-\frac{1}{\alpha}}\mathrm{B}^{-\frac{1+\alpha}{\alpha^2}}.$$
 Therefore, the constant $d$ is determined by taking
 $$ d=\max\left\lbrace 1,\,\left(\mathrm{C}(n)(1+ \mathrm{C}_2(2^{g_0+2}\mathrm{C}_3+ \mathrm{C}_3^2))\right)^{\frac{1}{2\alpha}}\cdot 4^{\left(\frac{1+\alpha}{\sqrt{2}\alpha}\right)^2} \right\rbrace. $$
 Finally, it follows from Lemma \ref{pre2} that 
  $$\lim_{i\to +\infty}\psi_i=0, $$
 and hence 
 $$v\leq 2d\quad\text{a.e.}\quad\text{in}\,\,\mathrm{B}_{\frac{1}{2}}^+. $$
 This ends the proof.
\end{proof}

\section{Hölder regularity of the gradient: Proof of Theorem \ref{T1}}\label{Section5}

With the foundational results in place, the next step involves a finer analysis of the region where the partial derivatives exhibit small magnitude, which is crucial for applying our estimates.

 \begin{lemma}\label{Lch}
  Let \(u\) be a minimizer in \(\mathrm{B}^{+}_{1}\) such that  \(\displaystyle\sup_{\mathrm{B}^{+}_{1/2}}|\nabla u| = 1\). Moreover, assume that the origin is a boundary point of the contact set, that is, \(0 \in \partial \{u = 0\} \cap \mathrm{T}_{1/2}\). Then the following statements hold:
 \begin{itemize}
  \item [(i)] There exists a positive constant \(\mathrm{C}_{5}\) such that 
\[
   \sup_{\mathrm{B}_{1/2}^{+}}|\nabla' u|\geq\mathrm{C}_{5}.
\]
 \item[(ii)] There exists a positive constant \(\mathrm{C}_{6}\) such that for any \(i=1,\ldots,n-1\), if 
 \[
 \displaystyle\sup_{\mathrm{B}^{+}_{1/2}}u_{x_{i}}=:\mathrm{M}_{i}\geq \frac{\mathrm{C}_{5}}{\sqrt{n-1}},
 \] 
 then
\[
 \int_{\mathrm{B}^{+}_{1/2}}|u_{x_{i}}-\mathrm{M}_{i}|dx\geq \mathrm{C}_{6}.
\]
 Moreover, the same result holds for \(-u\).
  \end{itemize}
 \end{lemma}

\begin{proof}
 We emphasize that the convergence arguments in both (i) and (ii) holds up to passing to a subsequence. Each of the statements (i) and (ii) will be proven through a \textit{reductio ad absurdum} argument. We begin by proving item (i). Suppose that there exists a sequence \((u^{k})_{k\in\mathbb{N}}\) of minimizers in the class $\mathcal{G}$ such that \(\displaystyle\sup_{\mathrm{B}^{+}_{1/2}}|\nabla u^{k}|=1\) and \(0\in \partial\{u^{k}=0\}\cap \mathrm{T}_{1/2}\), however,
 \begin{eqnarray}\label{convgrad'}
  \sup_{\mathrm{B}^{+}_{1/2}}|\nabla' u^{k}|\to 0\,\,\, \text{as}\,\,\, k\to \infty.
 \end{eqnarray}
 As \(0\) is a point of contact, the Lipschitz regularity of minimizers (Proposition \ref{Lipschitzregularity}) implies that \(u^{k}(0)=0\) and 
\[
\mathrm{c}_{0}\leq\|u^{k}\|_{C^{0,1}(\mathrm{B}^{+}_{1/2})}\leq \mathrm{C_1}\displaystyle\sup_{\mathrm{B}^{+}_{1/2}}|\nabla u^{k}|=\mathrm{C_1},
\]
 for constants \(\mathrm{C_1},\mathrm{c}_{0}>0\), independents of \(k\), where the existence of \(\mathrm{c}_{0}>0\) is ensured too, by the condition \(\sup_{\mathrm{B}^{+}_{1/2}}|\nabla u^{k}|=1\). 

 By considering the even extension of each function, we obtain a sequence $(\tilde{u}^k)_{k\in \mathbb{N}}$ in the class $\mathcal{\tilde{G}},$ which, as previously noted in \eqref{funcionaldoobstaculo}, are minimizers of the functional $\mathcal{\tilde{J}}.$ Since $\tilde{\varphi}\in \tilde{\mathcal{G}},$  it follows that $\tilde{\mathcal{J}}(\tilde{u}^k)\leq \tilde{\mathcal{J}}(\tilde{\varphi})$ for all $k\in \mathbb{N},$ and by Lemma \ref{poincareinequality}, we have
  $$\int_{\mathrm{B}_{1}}G(|\tilde{u}^k-\tilde{\varphi}|)\leq \int_{\mathrm{B}_{1}}G(\mathrm{C}|\nabla(\tilde{u}^k-\tilde{\varphi})|). $$
 Thus there exists a constant $\mathrm{C}>0$ such that for all $k\in \mathbb{N}$
 $$\int_{\mathrm{B}_{1}}G(|\tilde{u}^k|) + G(|\nabla \tilde{u}^k|)\,dx \leq \mathrm{C}. $$
 Therefore, since each $\tilde{u}_k$ is an even extension of $u_k,$ by Lemma \ref{inclusion}, we see that for all $k\in \mathbb{N}$ we have $||u^k||_{W^{1,G}(\mathrm{B}_{1}^+)}\leq \mathrm{C}.$ Consequently, there exists a function $u_{\infty}\in W^{1,G}(\mathrm{B}_{1}^+)$ such that, up to the extraction of a subsequence $u^k \rightharpoonup u_\infty$ in $W^{1,G}(\mathrm{B}_{1}^+).$ By Remark \eqref{trace}, we have that $u^k \rightharpoonup u_\infty$ in $W^{1,1+\delta_{0}}(\mathrm{B}_{1}^+).$  Moreover, by the compact embedding $W^{1,1+\delta_{0}}(\mathrm{B}_{1})\hookrightarrow L^{1+\delta_{0}}(\mathrm{B}_{1})$ we have $u^k\to u_{\infty}$ a.e. in $\mathrm{B}_{1}^+,$ and given that—as discussed in Remark \ref{trace}—the trace operator $T:W^{1,1+\delta_{0}}(\mathrm{B}_{1}^+)\to L^{r}(\partial \mathrm{B}_{1}^+)$ is compact, it follows that $u_\infty=\varphi$ on $(\partial \mathrm{B}_{1})^+$ and $u_\infty\geq 0$ on $\mathrm{T}_1.$ Thus, we conclude that $u_\infty\in \mathcal{G}.$ 
 Now, we observe that the convexity of $G$ implies that
  $$\int_{\mathrm{B}_{1}^+}G(|\nabla u^k|)\,dx \geq  \int_{\mathrm{B}_{1}^+}G(|\nabla u_\infty|)\,dx + \int_{\mathrm{B}_{1}^+}\frac{g(|\nabla u_\infty|)}{|\nabla u_\infty|}\nabla u_\infty\cdot(\nabla u^k- \nabla u_\infty)\,dx, $$
 and since $u^k \rightharpoonup u_\infty$ in $W^{1,G}(\mathrm{B}_{1}^+),$ it follows that
 $$\int_{\mathrm{B}_{1}^+}G(|\nabla u_\infty|)\,dx\leq \liminf_{k\to \infty} \int_{\mathrm{B}_{1}^+}G(|\nabla u^k|)\,dx.$$
  Hence, since $(u^k)_{k\in \mathbb{N}}$ are minimizers in $\mathcal{G},$ it follows that $u_\infty\in \mathcal{G}$ is also a minimizer of the same functional $\mathcal{J}.$ Thus, in particular $u_\infty$ is a $g$-harmonic function in $\mathrm{B}^+_{1/2}$. Moreover, the convergence in \eqref{convgrad'} implies that $\nabla' u_\infty = 0'$ in $\overline{\mathrm{B}_{1/2}^+}$. Hence, it is possible to check that \(u_{\infty}(x)=\mathrm{D}\cdot x_{n}\) in $\overline{\mathrm{B}_{1/2}^+},$ for some \(\mathrm{D}\in\mathbb{R}\setminus \{0\}\) since $\sup_{\mathrm{B}_{1/2}^{+}}|\nabla u^{k}|=1$ for all $k$. We claim that $\mathrm{D} < 0$. Indeed,
let us take a smooth nonnegative test function $\phi$ that vanishes in $\mathrm{T}_{1} \setminus \mathrm{T}_{1/2}$. In this case, by the weak formulation of the problem, we
\[
\int_{\mathrm{T}_{1/2}}\frac{g(|\nabla u^{k}|)}{|\nabla u^{k}|}u^{k}_{x_{n}}\phi dS\leq 0.
\]
By Fatou's Lemma, we have that
\[
\frac{g(|\mathrm{D}|)}{|\mathrm
D|}\mathrm{D}\int_{\mathrm{T}_{1/2}}\phi dS\leq \liminf_{k\to \infty}\int_{\mathrm{T}_{1/2}}\frac{g(|\nabla u^{k}|)}{|\nabla u^{k}|}u^{k}_{x_{n}}\phi dS\leq 0,
\]
implying that 
\[
\frac{g(|\mathrm{D}|)}{|\mathrm
D|}\mathrm{D}\leq0,
\]
which implies that \(\mathrm{D}<0\), since $g$ is a nonnegative function and $\mathrm{D}\neq 0$.
  
We now claim that there exists $k_{0}$ sufficiently large such that $u^{k} = 0$ in $\mathrm{T}_{1/2}$ for all $k \geq k_{0}$. Indeed, if $x = 0$, there is nothing to prove. On the other hand, since $u_{\infty}$ is negative in $\mathrm{B}^{+}_{1/2} \setminus \{x_{n} = 0\}$, by the uniform convergence $u^{k} \to u_{\infty}$, there exists some $k_{0} \in \mathbb{N}$ such that $u^{k}(x) \leq 0$ for all $x \in \mathrm{B}^{+}_{1/2} \setminus \{x_{n} = 0\}$ and all $k \geq k_{0}$. Thus, given $x \in \mathrm{T}_{1/2}$, by the continuity of $u^{k}$ for $k \geq k_{0}$, and using the hypothesis that $u^{k}$ is a minimizer of the thin obstacle problem, we obtain
$$
0 \leq u^{k}(x) = \lim_{\substack{y \to x \\ y \in \mathrm{B}^{+}_{1/2}}} u^{k}(y) \leq 0,
$$
and therefore $u^{k} = 0$ in $\mathrm{T}_{1/2}$ for all $k \gg 1$. Contradicting the fact that the origin $0$ lies on the boundary of $\{u^k=0\}.$ This contradiction establishes claim (i).

The next step is to establish item (ii). Assume that (ii) is false, then there exists a sequence of minimizers \((u^{k})_{k\in\mathbb{N}}\) such that
\[
\sup_{\mathrm{B}_{1/2}^+}|\nabla u^k| = 1, \quad 0 \in \partial\left(\{u^k=0\}\right)\cap \mathrm{T}_{1/2},
\]
  and for some \(i \in \{1,\dots,n-1\}\), we have the following
\[
\sup_{\mathrm{B}_{1/2}^+} (u^{k})_{x_{i}} =: \mathrm{M}_i^k \geq \frac{\mathrm{C}_0}{\sqrt{n-1}}, \quad \text{and} \quad \int_{\mathrm{B}_{1/2}^+} |(u^{k})_{x_{i}} - \mathrm{M}_i^k|^2\,dx \to 0.
\]
Observe that for every \(k \in \mathbb{N}\), we have
\[
\frac{\mathrm{C}_0}{\sqrt{n-1}} \leq \mathrm{M}_i^k \leq 1,
\]
so, there exists \(\mathrm{M}_i^0 \in \left[\frac{\mathrm{C}_0}{\sqrt{n-1}}, 1\right]\) such that \(\mathrm{M}_i^k \to \mathrm{M}_i^0\). Then, since
\[
\int_{\mathrm{B}_{1/2}^+} |(u^{k})_{x_{i}} - \mathrm{M}_i^0|^2 \leq 4\int_{\mathrm{B}_{1/2}^+} |(u^{k})_{x_{i}} - \mathrm{M}_i^k|^2 + 4 \int_{\mathrm{B}_{1/2}^+} |\mathrm{M}_i^k - \mathrm{M}_i^0|^2,
\]
it follows that \((u^{k})_{x_{i}} \to \mathrm{M}_i^0\) in \(L^2(\mathrm{B}_{1/2}^+)\). Since \(0 \in \partial(\{u^k = 0\}) \cap \mathrm{T}_{1/2}\), we have \(u^k(0) = 0\). Hence, for every \(x \in \mathrm{B}_{1/2}^+\),
\[
|u^k(x)| \leq \|\nabla u^k\|_{L^\infty(\mathrm{B}_{1/2}^+)}\cdot |x| \leq |x|,
\]
and for all \(k \in \mathbb{N}\),
\begin{equation}\label{lim1}
\|u^k\|_{W^{1,\infty}(\mathrm{B}_{1/2}^+)} \leq \mathrm{C}.
\end{equation}
Thus, for every \(q \in (1,\infty)\), we obtain \(\|u^k\|_{W^{1,q}(\mathrm{B}_{1/2}^+)} \leq \mathrm{C}(n,q)\). By reflexivity, there exists \(u^0 \in W^{1,q}(\mathrm{B}_{1/2}^+)\) such that
\[
u^k \rightharpoonup u^0 \quad \text{in } W^{1,q}(\mathrm{B}_{1/2}^+).
\]
In particular, taking \(q = 2\), the uniqueness of the weak limit yields
\[
(u^{0})_{x_{i}} \equiv \mathrm{M}_i^0.
\]
Observe that each \(u^k \geq 0\) on \(\mathrm{T}_1\). By the Trace Compactness Theorem as in Remark \ref{trace}, we conclude that \(u^0 \geq 0\) on \(\mathrm{T}_{1/2}\) as well. Furthermore, by \eqref{lim1} and the Arzelà–Ascoli theorem, there exists a function \(w\) such that
\[
u^k \to w \quad \text{uniformly in } \mathrm{B}_{1/2}^+.
\]
Since \(u^k(0) = 0\), we deduce \(w(0) = 0\), and by uniqueness of the limit, \(u^0 = w\); in particular, \(u^0(0) = 0\). Let \(v(x') := u^0(x', 0)\). Then \(x' = 0\) is a minimum point of \(v\), hence \(\nabla v(0) = 0\). Thus, the vanishing of \(\nabla' u^0(0)\) implies that \(\partial_i u^0(0) = 0\).
But this is in contradiction with the fact that we had established
\[
 (u^0)_{x_{i}} \equiv \mathrm{M}_i^0 \geq \frac{\mathrm{C}_0}{\sqrt{n-1}} > 0,
\]
Note that the same argument applies to the sequence \((-u^k)_{k \in \mathbb{N}}\), since in that case we also have \(\sup_{\mathrm{B}_{1/2}^+}|\nabla(-u^k)| = 1\), \(0 \in \partial(\{-u^k = 0\}) \cap \mathrm{T}_{1/2}\), and \(-u^k \leq 0\) on \(\mathrm{T}_1\). Hence, the same conclusion holds for \(-u\).
\end{proof}

 The following lemma can be found in \cite{AM}.
\begin{lemma}\label{aux}
Let $\omega, \omega_1,...,\omega_\mathrm{N} \geq 0$ non decreasing functions on an interval $(0, R_0)$ such that for each $R\leq R_0$ there is a function $\omega_k$ satisfying 
$$
\omega(R)\leq \delta\omega_k(R)\quad\text{and}\quad \omega_k(R/2)\leq (1-\lambda)\omega_k(R),
$$
for constants $\delta>0,$ and $\lambda\in (0,1).$ Then,	
$$
\omega(R)\leq \mathrm{C}\left(\frac{R}{R_0}\right)^{\beta},
$$
where $\beta=\beta(\mathrm{N},\delta, \lambda)$ and $\mathrm{C}=\mathrm{C}(\mathrm{N},\delta, \gamma)\max\limits_{1\leq k\leq \mathrm{N}}\omega_k(R_0).$ 
\end{lemma}

We are now in a position to present the proof of the main result.
\begin{proof}[{\bf Proof of the Theorem \ref{T1}}]
Once the previous results for the operator \(g\)-Laplacian have been established, the argument follows exactly the same strategy as in Theorem 4.3 of \cite{AM}, by combining these ingredients. For the sake of clarity and as a courtesy to the reader, we include the full proof here. Assume that $0\in \partial\left(\{u=0\}\right)\cap \mathrm{T}_{1}$.
  First, we want to show that for each $\rho\in (0, 1/2)$ there is some $i=1,...,n-1$ for which holds 
  $$\sup_{\mathrm{B}_{\rho/4}^+}u_{x_i} \leq (1-\lambda)\sup_{\mathrm{B}_{\rho/2}^+}u_{x_i}\quad\text{or}\quad \sup_{\mathrm{B}_{\rho/4}^+}(-u)_{x_i} \leq (1-\lambda)\sup_{\mathrm{B}_{\rho/2}^+}(-u)_{x_i}, $$
  for an universal constant $\lambda\in (0,1).$ Let's consider the function given by a scaling and a normalization of the minimizer $u,$ i.e
  $$v(x):= \frac{u(\rho x)}{\rho \mathrm{K}}\quad\forall x\in \mathrm{B}_{1}^+, $$
  where $\mathrm{K}=\sup_{\mathrm{B}_{\rho/2}^+}|\nabla u|. $ By the Remark \ref{scaling} $v$ satisfies the hypotheses of the Lemma \ref{Lch} therefore there exists a constant $\mathrm{C}_5>0$ such that
  $$\sup_{\mathrm{B}_{1/2}^+}|\nabla' v|\geq \mathrm{C}_5. $$
  Moreover, there exists a positive constant $\mathrm{C}_6$ such that if for some $i=1,2,..,n-1$ one has
  \begin{equation}\label{Eq12}
      \sup_{\mathrm{B}_{1/2}^+}v_{x_i}\geq \frac{\mathrm{C}_5}{\sqrt{n-1}},
  \end{equation}
  then
  $$\int_{\mathrm{B}_{1/2}^+}|v_{x_i}- \mathrm{M}_{i}|^2 \geq \mathrm{C}_6,\quad\text{where}\,\,\,\mathrm{M}_{i}=\sup_{\mathrm{B}_{1/2}^+}v_{x_i}. $$
 The same is true for $-v.$ Therefore, if we define the set
 $$\mathrm{A}(k)=\{x\in \mathrm{B}_{1}^+\;\, v_{x_i}(x)>k \}, $$
 we get for constants $k_0\in (0, \mathrm{M}_{i})$ sufficiently close to $\mathrm{M}_{i}$ and $\varrho_1\in (0,1)$ sufficiently close to $1,$ we obtain that
  $$|\mathrm{A}(k_0)\cap \mathrm{B}_{1/2}^+| \leq \varrho_1 |\mathrm{B}_{1/2}^+|. $$
 Applying the Lemma \ref{degeo2}, we have for a constant $\mathrm{C}_0$ sufficiently large to be choose later, and a constant $k\in (0,1)$ that depends on $\mathrm{C}_0$ so that  
 \begin{equation}\label{Eq11}
     |\mathrm{A}(k\mathrm{M}_{i})\cap \mathrm{B}_{1/2}^+| \leq \frac{|\mathrm{B}_{1/2}^+|}{\mathrm{C}_0}.
 \end{equation}
 Applying Lemma \ref{degeo3} we get
 \begin{equation}\label{Eq10}
 \sup_{\mathrm{B}_{1/4}^+}v_{x_i} \leq \mathrm{C}_4\left(\int_{\mathrm{B}_{1/2}^+}((v_{x_i}- \mathrm{M}_{i}k)^+)^2\,dx\right)^{1/2}|\mathrm{A}(\mathrm{M}_{i}k)\cap \mathrm{B}_{1/2}^+|^{\frac{\alpha}{2}} + \mathrm{M}_{i}k.
 \end{equation}
 Since $v_{x_i}\leq \mathrm{M}_{i}$ in $\mathrm{B}_{1/2}^+,$ then
  $$\mathrm{C}_4\left(\int_{\mathrm{B}_{1/2}^+}((v_{x_i}- \mathrm{M}_{i}k)^+)^2\,dx\right)^{1/2}\leq \mathrm{C}_4(1-k)\mathrm{M}_{i}|\mathrm{B}_{1/2}^+|^{1/2}. $$
 Combining \eqref{Eq11} and \eqref{Eq10} we have
 $$\sup_{\mathrm{B}_{1/4}^+}v_{x_i} \leq \frac{\mathrm{C}_4|\mathrm{B}_{1/2}^+|^{\frac{1+\alpha}{2}}}{\mathrm{C}_0^{\alpha/2}}(1-k)\mathrm{M}_{i} + \mathrm{M}_{i}k. $$
 Hence, we choose $\mathrm{C}_0$ to be sufficiently large in order to ensure that
 $$\frac{\mathrm{C}_4|\mathrm{B}_{1/2}^+|^{\frac{1+\alpha}{2}}}{\mathrm{C}_0^{\alpha/2}}= \frac{1}{2}. $$
 It follows that
  $$\sup_{\mathrm{B}_{1/4}^+}v_{x_i} \leq (1-\lambda)\cdot\sup_{\mathrm{B}_{1/2}^+}v_{x_i}, $$
 with $\lambda= \frac{1-k}{2}.$ Recalling the definition of $v,$ it follows from \eqref{Eq12} that if, for some $i=1,...,n-1$ the inequality holds 
 $$\sup_{\mathrm{B}_{\rho/2}^+}u_{x_i} \geq \frac{\mathrm{C}_5}{\sqrt{n-1}}\cdot\sup_{\mathrm{B}_{\rho/2}^+}|\nabla u|, $$
 then
 $$\sup_{\mathrm{B}_{\rho/4}^+}u_{x_i} \leq (1-\lambda)\cdot\sup_{\mathrm{B}_{\rho/2}^+}u_{x_i}. $$	
 The same holds for $-u.$

 Now, using Lemma \ref{Lch} we have
  $$\sup_{\mathrm{B}_{\rho/2}^+}|\nabla' u| \geq \mathrm{C}_5\cdot\sup_{\mathrm{B}_{\rho/2}^+}|\nabla u|,$$
 in particular, there exists some $i=1,...,n-1$ such that it holds that
 $$\sup_{\mathrm{B}_{\rho/2}^+}u_{x_i} \geq \frac{\mathrm{C}_5}{\sqrt{n-1}}\cdot\sup_{\mathrm{B}_{\rho/2}^+}|\nabla u|\quad\text{or}\quad \sup_{\mathrm{B}_{\rho/2}^+}(-u)_{x_i} \geq \frac{\mathrm{C}_5}{\sqrt{n-1}}\cdot\sup_{\mathrm{B}_{\rho/2}^+}|\nabla u|.$$
 Hence, from the previous arguments, it follows that
 $$\sup_{\mathrm{B}_{\rho/4}^+}u_{x_i}\leq (1-\lambda)\sup_{\mathrm{B}_{\rho/2}^+}u_{x_i}\quad\text{or}\quad \sup_{\mathrm{B}_{\rho/4}^+}(-u)_{x_i}\leq (1-\lambda)\sup_{\mathrm{B}_{\rho/2}^+}(-u)_{x_i}. $$
In order to apply Lemma \ref{aux}, we choose $\delta= \frac{\sqrt{n-1}}{\mathrm{C}_5},\,\,\omega(\rho)=\sup_{\mathrm{B}_{\rho/2}^+}|\nabla u|,\,\,\omega_k(\rho)=\sup_{\mathrm{B}_{\rho/2}^+}u_{x_k},$ for $1\leq k\leq n-1$ and $\omega_k(\rho)=\sup_{\mathrm{B}_{\rho/2}^+}(-u)_{x_k}$ for $n\leq k\leq 2n-2$. In this case, we can conclude that
 $$\sup_{\mathrm{B}_{\rho}^+(0)}|\nabla u| \leq \mathrm{C}_7\cdot\rho^{\beta}, $$
 where $\mathrm{C}_7:=C(n,\lambda,\delta)\cdot\sup_{\mathrm{B}_{3/4}^+}|\nabla'u|$ and $\beta=\beta(n,\delta, \lambda)>0.$

 Now for the sets $\Sigma:= \mathrm{B}_{\frac{1}{4}}\cap \{x_n=0\}\cap \{u=0\}$ and $\Gamma:=\mathrm{B}_{\frac{1}{4}}\cap \{x_n=0\}\cap \{u>0\},$ we have 
 $$\partial\left(\{u=0\}\right)\cap \mathrm{B}_{\frac{1}{4}}\cap\{x_n=0\}= \overline{\Sigma}\cap \overline{\Gamma}.$$
 Let $x\in \mathrm{B}_{\frac{1}{4}}^+,$ and $\overline{x}\in \overline{\Sigma}\cap \overline{\Gamma}$ such that
 $$\operatorname{dist}\left(x, \overline{\Sigma}\cap \overline{\Gamma}\right)= |x-\overline{x}|. $$
 The same argument used previously is now applied to the half-ball $\mathrm{B}_{\rho}^+(\overline{x})$ with $\rho= |x-\overline{x}| + \epsilon,$ so we get
  \begin{equation}\label{desig prin}
	|\nabla u(x)| \leq \mathrm{C}_7\cdot \operatorname{dist}\left(x, \overline{\Sigma}\cap \overline{\Gamma}\right)^{\beta}\quad\forall x\in \mathrm{B}_{\frac{1}{2}}^+.
  \end{equation}
 Now, for any points $x_1, x_2\in \mathrm{B}_{\frac{1}{4}}^+,$ define
 $$d_i:=\dist\left(x_i, \overline{\Sigma}\cap \overline{\Gamma}\right)\quad\text{for}\quad i=1,2.$$
 Moreover, let $\overline{x_i}\in \overline{\Sigma}\cap\overline{\Gamma}$ such that $d_i= |x_i-\overline{x_i}|$ for $i=1,2.$ Without loss of generality, we may assume that $d_1\leq d_2.$ Applying the triangle inequality, we obtain the following inequalities: 
 $$d_2\leq |x_2-x_1| + d_1.$$ 
 There are three cases to consider in what follows:\\  {\bf 1°Case:} $d_2\leq 4|x_1-x_2|.$  \\
 Under this condition, we may apply  (\ref{desig prin}) to derive the following
		\begin{align*}
			|\nabla u(x_1)-\nabla u(x_2)| &\leq \mathrm{C}_7\left(d_1^{\beta}+ d_2^{\beta}\right)\\			
			&\leq 2^{2\beta+1}\mathrm{C}_7\cdot|x_1-x_2|^{\beta}.
		\end{align*} 
{\bf 2°Case:} $d_1\leq 4|x_1-x_2|\leq d_2.$\\
Following the same approach as in the previous case, it becomes evident that
      $$|\nabla u(x_1)-\nabla u(x_2)|\leq 2^{3\beta +1}\mathrm{C}_7\cdot|x_1-x_2|^{\beta}. $$
 {\bf 3°Case:} $4|x_1-x_2| < d_1\leq d_2.$\\
Observe that, in this case, the ball $\mathrm{B}_{\frac{d_1}{2}}\left(\frac{x_1+x_2}{2}\right)$ cannot intersect both $\Sigma$ and $\Gamma.$ If the ball $\mathrm{B}_{\frac{d_1}{2}}\left(\frac{x_1+x_2}{2}\right)$ intersects $\Sigma,$ we consider the odd reflection of $u$ given by
      	$$
		   	\overline{u}(x',x_n)=
		   	\begin{cases}
		   		u(x',x_n);\quad\text{if}\quad x_n\geq 0\\
		   		-u(x',-x_n);\quad\text{if}\quad x_n\leq 0.\\
		   	\end{cases}
	 	$$
    By the Schwarz Reflection Principle (cf. \cite[Proposition 2.1]{BM21}), we have that $\tilde{u}$ is $g-$harmonic in $\mathrm{B}_{\frac{d_1}{2}}\left(\frac{x_1+x_2}{2}\right)$. In the case where, that ball intersects $\Gamma,$ we consider the even reflection $\tilde{u}$ which was previously used. In this case, it follows that $\tilde{u}$ is $g-$harmonic in $\mathrm{B}_{\frac{d_1}{2}}\left(\frac{x_1+x_2}{2}\right)$. Therefore, up to considering either the odd or even reflection, we have that 
    $$\Delta_gu=0\quad\text{in}\quad \mathrm{B}_{\frac{d_1}{2}}\left(\frac{x_1+x_2}{2}\right).$$
    Now we define the function
    $$w(x):= \frac{u\left(d_1/2\cdot x + \frac{x_1+x_2}{2}\right)}{d_1/2\cdot d_1^\beta}\quad \forall x\in \mathrm{B}_1(0). $$
    Since $w$ is obtained via normalization and scaling as in Remark \ref{scaling}, it follows that $w$ is $g^*-$harmonic in $\mathrm{B}_1(0).$
    Then, the interior $C^{1,\alpha}$ regularity theory for $g$-harmonic functions (cf. \cite[Theorem 1.7]{Lie91}) provides that there exist constants $\mathrm{C}_8, \theta$ depending on $\delta_0, g_0$ and $n$  such that
    \begin{equation}\label{Eq14}
        \operatornamewithlimits{osc}_{\mathrm{B}_{\frac{|x_1-x_2|}{d_1}}}|\nabla w| \leq 4^\theta \mathrm{C}_8\cdot\sup_{\mathrm{B}_{1/4}}|\nabla w|\left(\frac{|x_1-x_2|}{d_1}\right)^\theta.
    \end{equation}
    From inequality \eqref{desig prin}, we see that
     $$\sup_{\mathrm{B}_{\frac{d_1}{4}}\left(\frac{x_1+x_2}{2}\right)}|\nabla u| \leq \mathrm{C}_7(4d_1)^{\beta}. $$
    Applying this estimate in \eqref{Eq14} we get
    $$
    |\nabla u(x_1)- \nabla u(x_2)|\leq 4^{\theta+\beta} \mathrm{C}_8  \mathrm{C}_7 d_1^{\beta}\left(\frac{|x_1-x_2|}{d_1}\right)^\theta. 
    $$
    Therefore, if $\beta\geq \theta$ we obtain the desired estimate; otherwise, since $\frac{|x_1-x_2|}{d_1}<1$ the inequality above implies that
    $$|\nabla u(x_1)- \nabla u(x_2)|\leq 4^{\theta+\beta} \mathrm{C}_8  \mathrm{C}_7|x_1-x_2|^\beta. $$
  Thus, by setting $\gamma=\min\{\beta, \theta\},$  we conclude that in either case the following estimate holds 
  $$|\nabla u(x_1)- \nabla u(x_2)|\leq 4^{\theta+\beta}2^{3\beta +1} \mathrm{C}_8  \mathrm{C}_7|x_1-x_2|^{\gamma}.$$

Finally, observe that the argument presented can be applied at any ball \(\mathrm{B}_r(z) \subset \mathrm{B}_{3/4}\) centered at a point \(z \in \mathrm{B}_{1/2} \cap \{x_n = 0\}\). Using a covering argument along with the $C^{1,\alpha}$ estimate for $g-$harmonic functions, we conclude the result, possibly with new constants larger than the previous ones. 

\end{proof} 

\subsection*{Acknowledgments}
 J. da Silva Bessa was supported by FAPESP-Brazil under Grant No. 2023/18447-3. P. H. da Costa Silva has been supported by CAPES-Brazil under Grant No. 88881.126989/2025-01.

\subsection*{Conflict of interest}

On behalf of all authors, the corresponding author states that there is no conflict of interest

\subsection*{Data availability statement}

Data availability does not apply to this article as no new data were created or analyzed in this study.

\end{document}